\documentclass[11pt]{amsart}
\usepackage{amssymb}
\usepackage{graphicx}
\usepackage{hyperref}

\usepackage{xcolor} 
\usepackage{tensor}
\usepackage{fullpage} 

\definecolor{green}{rgb}{0,0.8,0} 

\newtheorem{theorem}{Theorem}[section]

\newtheorem{lemma}[theorem]{Lemma}
\newtheorem{proposition}[theorem]{Proposition}
\theoremstyle{definition}

\theoremstyle{remark}
\newtheorem{remark}[theorem]{Remark}
\numberwithin{equation}{section}

\newcommand{\abs}[1]{\vert#1\vert^2}

\newcommand{\pa}{\partial}
\newcommand{\na}{\nabla}

\newcommand{\Del}{\Delta}
\newcommand{\ep}{\epsilon}

\newcommand{\la}{\lambda}

\newcommand{\Om}{\Omega}
\newcommand{\ps}{\psi}

\newcommand{\bbR}{\mathbb R}
\newcommand{\bbr}{\mathbb R}

\newcommand{\calP}{\mathcal P}

\newcommand{\PP}{\mathcal P}

\newcommand{\ddt}{\frac{d}{dt}}
\newcommand{\vp}{\varphi}
\newcommand{\dvg}{\na \cdot}
\newcommand{\dv}{\na \cdot}

\newcommand{\be}{\begin{equation}}
\newcommand{\ee}{\end{equation}}
\newcommand{\ben}{\begin{equation*}}
\newcommand{\een}{\end{equation*}}
\newcommand{\bsp}{\begin{split}}
\newcommand{\esp}{\end{split}}

\newcommand{\fn}[1]{\frac{ #1}{N}}

\newcommand{\fnn}[1]{\frac{ #1}{\sqrt{N}}}

\newcommand{\lr}[1]{\left( #1 \right)}
\newcommand{\ipa}{\pa_y^{i}}
\newcommand{\jpa}{\pa_z^{j}}
\newcommand{\Fn}{\left(\fn{1} \right)}

\begin{document}

\title[]{Nonlinear stability of planar traveling waves\\ 
in a chemotaxis model of tumor angiogenesis
 \\   with chemical diffusion }


\author{Myeongju Chae}%
\address{Department of Mathematics, Hankyong University, Anseong-si, Gyeonggi-do, Republic of Korea}%
\email{mchae@hknu.ac.kr}
\author{Kyudong Choi}%
\address{Department of Mathematical Sciences, Ulsan National Institute of Science and Technology, Ulsan, Republic of Korea}%
\email{kchoi@unist.ac.kr}

\date{\today}%
\begin{abstract}
We consider
a simplified  chemotaxis model of   tumor angiogenesis, described by  a  Keller-Segel system 
on  the two dimensional infinite cylindrical domain
$ (x, y) \in  
  \bbr \times {\mathbf S^{\la}}$, 
where $ \mathbf S^{\la}$ is the circle of perimeter $\la>0$. The domain  models a virtual channel where newly generated blood vessels toward the vascular endothelial  growth factor  will be located.
The system is known to allow planar traveling wave solutions of an invading type. In this paper, we establish the nonlinear stability of these traveling invading waves  when   chemical diffusion     is present if  $\la$ is sufficiently small.
The same result for the corresponding system in one-dimension was  obtained by  Li-Li-Wang (2014) \cite{LiLiWa}. Our result solves the problem remained open in  \cite{CCKL} at  which only linear stability of  the waves was obtained  under certain artificial assumption. 
\end{abstract}
\maketitle
\section{Introduction}\label{sec_intro}
\subsection{A Keller-Segel system}

The formation of new blood vessels from pre-existing vessels, which is so-called angiogenesis, is the essential mechanism for tumour progression and metastasis. Focusing on the interaction between endothelial cells and growth factor, a simplified model of   tumor angiogenesis  can be described by the following Keller-Segel system \cite{FrTe,Le,Pe}:
 \begin{align}\label{KS} \begin{aligned}
\pa_t n - \Del n& = - \na \cdot (n \chi(c) \na c ) \\
\pa_t c - \ep\Del c & = - c^m n. 
\end{aligned} \end{align} 
\indent We consider the above system in two-dimension   with a front boundary condition in $x$   and a periodic condition in $y$, both specified later,
with $m >0$ and $\ep>0$.  
In a general Keller-Segel context, the unknown  $n(x, y, t)> 0 $ is the bacterial density while
the unknown  $c(x,y, t)> 0$ is the concentration of chemical nutrient consumed by bacteria at position $(x,y)$ and time $t$. Considering formation of new blood vessels, $n$ denotes the density of endothelial cells while $c$ does the concentration of  the protein known as the vascular endothelial  growth factor (VEGF).   
The chemosensitivity function  $\chi(\cdot): \bbr^+ \to  \bbr^+ $ is a given decreasing    function, 
reflecting that the chemosensitivity gets lower as the  concentration of the chemical gets higher. 
 The positive constant $\epsilon>0$ is the diffusion rate constant for the chemical substance  $c$ while 
$m$ indicates  the consumption rate of nutrient $c$.  
 
\indent
 When we model  endothelial angiogenesis,   we interpret that the  endothelial cells
behave as an invasive species, responding to signals produced by the hypoxic tissue.  Accordingly, we choose the $x$-axis by the propagating direction and 
 the system \eqref{KS} is given the front condition at left-right ends such that
\begin{align}\label{nfront}
\lim_{x\to -\infty} n (x, y, t) = n_- >0 ,  \quad  \lim_{x\to \infty} n(x, y, t) = 0,\\
\lim_{x\to -\infty} c (x, y, t) = 0,  \quad  \lim_{x\to \infty} c(x, y, t) = c_+ > 0. \label{cfront}
\end{align} To all functions in this paper, we impose  the periodic condition in $y$-variable of  period $\la>0$.

\indent
A \textit{planar}  traveling wave solution of \eqref{KS} is a traveling wave solution 
independent of the transversal direction $y$: 
\begin{align}\label{twave}
 n(x,y,t) =N(x-st), \quad c(x, y, t) = C(x-st)
\end{align}
with  a given wave speed $s>0$ which we always assume positive  in this paper without loss of generality.    We consider only  waves $(N, C)$ satisfying the boundary conditions \eqref{nfront} and \eqref{cfront} which means
\begin{equation}\label{bdry_nc}
\lim_{x\to -\infty}  N (z) =n_->0,\quad \lim_{x\to + \infty}  C  (z) =c_+>0,\quad \lim_{x\to +\infty}  N(z) = \lim_{x\to - \infty}  C (z) =0.
\end{equation}
We also assume  that
\begin{equation}\label{bdry_derivative_nc}
\lim_{x\to \pm \infty}  N'(z) = \lim_{x\to \pm \infty}  C' (z) =0.
\end{equation}
\noindent
To have a traveling wave, it is known that 
  that the chemosensitivity function $\chi( c)$  needs to be singular near  $c=0$  (\textit{e.g.} see \cite{KSb, Sc}). In the paper \cite{KSb}, $\chi(c)= c^{-1}$, which yields  the logarithmic singularity $(\chi(c)\nabla c) = \nabla\ln(c)$, is assumed, which choice of  $\chi(c)$  is then adopted on modeling the formation of the vascular network toward cancerous cells (\textit{e.g.} see \cite{FrTe,Le,Pe}). The existence of traveling wave solution with an invading front might  be an evidence of the tumor encapsulation (\textit{e.g.} see \cite{AC, BT, Sh}).\\

\indent
In this paper,  we consider only the case    $\chi(c)= c^{-1}$ and $m =1$ of \eqref{KS}:
\begin{align}\label{eq:main}
 \begin{aligned}
\pa_t n - \Del n& = - \na \cdot \left(n \frac{ \na c}{c} \right), \\
\pa_t c-\ep\Del c  & = - cn, \quad (x, y,t) \in \bbr\times  \mathbf{S}^{\la}\times  \bbr_+
\end{aligned} \end{align} where $ \mathbf{S}^{\la}
$ is the circle of perimeter $\la>0.$ This 2D cylindrical domain  would be understood as  a virtual channel where newly generated blood vessels toward the chemical (VEGF)   will be located. We   focus on establishing 
  the time asymptotic stability of a planar traveling wave solution $(N, C)$ of \eqref{eq:main}.
 The restriction on $m =1$ is required for  treating the singularity of $ 1/c$ by the Cole-Hopf transformation
\begin{equation}
\label{CH}  p := - \na \ln c = -\frac{\na c}{c}= -(\frac{\pa_x  c}{c},\frac{\pa_y c}{c}).
\end{equation}    A  well-written   explanation of the system including biological interpretation can be found  in \cite{Rosen1, Rosen2} (also refer to  \cite{PeWa} and the references therein).

\subsection{A parabolic system of conservation laws}
By  the Cole-Hopf transform, we translate the singular Keller-Segel system \eqref{eq:main} into the following system of $(n, p)=(n, (p_1,p_2))$  without singularity:
 \begin{align}\label{nq} \begin{aligned}
&\pa_t n -\Del n  =  \na \cdot (np) , \\
&\pa_t p -  \ep \Del p   =  -2\ep (p \cdot \na ) p +\na n, \quad (x, y,t) \in \bbr\times  \mathbf{S}^{\la}\times  \bbr_+
\end{aligned} \end{align}
with the notation $ ((p\cdot \na)p )_i = \sum_{k=1,2} p_k \pa_k p_i$.

By denoting  
 \begin{equation}\label{CP} \calP:=-C'/C\quad\mbox{and } P:=(\calP,0),
\end{equation} 
we have a planar traveling wave solution $(N, P)=(N, (\calP,0))$ of \eqref{nq}   of speed $s$ with the boundary conditions  inherited from those of $(N, C)$. The existence and some properties of those waves $(N,C)$ and   $(N, \calP)$ can be found in  \cite{Wa, LiWa}. We put some of the results on the waves we need in Subsection \ref{subsec_existence_traveling}.

\indent
 The study on  the existence of traveling wave solutions of a Keller-Segel model was  initiated by the paper \cite{KSb} then many works followed (see \cite{Ho1} and the references therein). {We also refer to  the    survey paper \cite{Wa_survey} which is an excellent exposition of the topic.}
 The existence of  traveling waves 
 with the  front conditions  \eqref{nfront} and \eqref{cfront} 
can be found
 in \cite{WaHi} for $\ep=0$,  and \cite{LiWa, Wa} for $\ep>0$.  
When considering the one dimensional system (i.e. no $y$-dependency in \eqref{eq:main}), the nonlinear   stability results were shown 
 in a weighted Sobolev space
  in \cite{JinLiWa} for $\ep=0$, and \cite{LiLiWa} when $\ep>0$ is small (also see \cite{PeWa}).  The weighted Sobolev space has commonly appeared when studying nonlinear stability of viscous shocks of conservations laws since 
  \cite{KaMa} (also see \cite{Ni}).
  
  \indent The study of higher dimensional traveling waves   is a very interesting topic and  remains open for many questions including existence and stability of such waves as indicated in  
\cite{Wa_survey}. As a special case in 2D,   planar waves  for an   infinite cylinder $\mathbb{R}\times \mathbf{S}^\la$ was considered  by \cite{CCKL} following the spirit of  the nonlinear energy estimates  developed in \cite{JinLiWa} for the whole line $\mathbb{R}$ case.  In angiogenesis, one may consider that a blood vessel in our body has a 2D cylinder structure. 

\indent
   {The previous result \cite{CCKL} mainly proved two things: one is the nonlinear stability for $\ep=0$   and the other is  the  stability of the linearized equation for small $\ep>0$    under the additional  mean-zero assumption   in transversal direction $y$ for some technical reason.  In addition to these two results, Theorem 1.6 in \cite{CCKL} gives  a hint  why studying planar waves is natural instead of doing general 2D traveling waves  by showing that the $y$-derivative of any smooth solution $(n,p)$ decays to zero in $L^2$-sense under certain additional assumption.}

\indent
In this paper, we show that  traveling wave solutions $(N, P)$ of {the nonlinear system  \eqref{nq} are globally stable under the smallness assumption on the parameters  $\ep>0$ and  $\la>0$ without the artificial  mean-zero assumption   in transversal direction $y$, which was needed in \cite{CCKL} even for the corresponding linearized system. Indeed, the main estimate \eqref{main_est}  holds uniformly for small $\ep>0$  when the antiderivative $(\vp,\psi)$ of a perturbation of the form $(n-N,p-P)=(\dv\vp,\na\psi)$ } is sufficiently small   in a weighted Sobolev space (see \eqref{np_perturb} and \eqref{w_sobolev}). 
Our result can be considered as an extension of \cite{CCKL} into $\ep>0$ case and an extension of \cite{LiLiWa} into 2D case.
See Theorem \ref{theoremnc} and Subsection \ref{subsec_main_thm} for the precise set-up. 
We state the stability result in terms of the perturbation of $(n, p)$ in Theorem \ref{theoremnc}, then explain the implication of the theorem for the perturbation of $(n,c)$ to $(N, C)$ in Remark \ref{remark_nc}.

 \indent
 At first glance, the transformed $(n, p)$-system \eqref{nq} seems simpler than the $(n, c)$-system \eqref{eq:main}
 to analyze since this parabolic system \eqref{nq} of conservation laws does not  have the logarithmic singularity. As a price for this, however,  the quadratic nonlinear term   $2\ep (p \cdot \na ) p$ appears, 
 and it is not clear at all if 
{the linear term $2\ep P\cdot\na\psi$ in the main perturbation equation \eqref{main_eq} produced by the nonlinear term  $2\ep (p \cdot \na ) p$ in \eqref{nq}
  can be controlled by the diffusion term  $\ep\Del\psi$  in \eqref{main_eq} produced by the diffusion term $\ep \Del p$ in \eqref{nq}}.
  
  \indent {In this regard, the main obstacle   is to handle the quantity $\ep\int_0^t \|\sqrt{\calP'}\psi\|^2$
 in \eqref{eq_lemma0_-}, which is  the time integral of a  localized $L^2$-norm of $\psi$ multiplied by $\ep$.  We overcome the difficulty thanks to {certain } dissipations of a localized $L^2$-norm of $\vp$ (not of $\psi$)  together with a careful manipulation done in Lemma \ref{lemma0_temp} (see Remark \ref{rem_difficulty}). In doing so, we need   the smallness assumption on  $\ep>0$.} This idea was first used in \cite{LiLiWa} for the  one-dimensional system while for  our  two dimensional system, it becomes more delicate due to the non-symmetric nature of the main perturbation equation \eqref{main_eq} on the propagating {direction $x$} and {the} transversal {direction $y$}. { For instance, when we denote $\vp=(\vp^1,\vp^2)$, we see the non-symmetric term
$\int\frac{\PP}{N}\vp^1(\vp^2)_y$ in  \eqref{quad_control_0-}. The smallness condition on the chemical diffusion constant $\ep$ might be understood in the sense that  
   the chemical in angiogenesis often diffuses in the dense    network of extracellular matrix and tissues  which are almost static as mentioned in \cite{Wa_survey}.}
 
  \indent {
  Unfortunately, we also need the smallness assumption on the perimeter $\la>0$ of a 2D infinite cylinder, and it appears due to a technical reason in our proof. In fact, with wave speed  $s>0$,  we ask   the product $s\cdot\la$ to be smaller than a given absolute constant (see \eqref{smallness_lemma0_-_la}). This condition enables us to employ Poincar\'{e} inequality \eqref{ineq:poincare2} in the transversal direction $y$ in order to control a non-localized   $L^2$-norm of $\vp$ (see \eqref{quad_control_0-} and \eqref{use_poincare})). In our opinion, it is very challenging      to remove this technical  smallness assumption  on $\la>0$.}

   \indent
  For the Cauchy problem of \eqref{KS}, we refer to \cite{CPZ1, CPZ2, FoFr, FrTe, LiLi}, where \cite{CPZ1, CPZ2} prove the existence of a global weak solution, and \cite{LiLi} proves the existence of  a global classical solution considering  the zero chemical diffusion case in a multi-dimension.  When a bounded domain is considered, a boundary layer may appear. We refer to \cite{HoLiWa} and \cite{PeWaZhZh} for the stability questions of the layer.

\indent
The remaining parts of the paper are organized as follows. In Subsection \ref{subsec_existence_traveling},
 we introduce background materials including the existence and some  properties of traveling wave solutions and state the main result (Theorem \ref{theoremnc}) with its set-up in Subsection \ref{subsec_main_thm}.   In Subsection \ref{hopf}, we  state  the local existence of a perturbative solution and  its  \textit{a priori} uniform-in-time estimate. In Section \ref{sec_three_main}, we prove the uniform-in-time   estimate. The zero-th and first order estimates (Subsection \ref{sec_zero} and \ref{sec_first}) are  the essential part. Then,  the higher order estimate (Lemma \ref{lemma23_}) can be obtained in a similar way. We present its proof  for completeness in Subsection \ref{sec_higher}.

\section{Main theorem and background materials}
\begin{subsection}{Existence and Properties of traveling wave solutions}\label{subsec_existence_traveling}

\indent We collect some results on   traveling wave solutions $(N, C)$ and $(N,P)=(N, (\calP,0))$ with the  front conditions introduced in Section \ref{sec_intro}.

We first observe that a traveling wave solution $(N, C)$ defined by \eqref{twave}  solve the following ODE system by plugging the expression \eqref{twave} into \eqref{eq:main}:
\begin{align} \begin{aligned} \label{NC}
-sN' -  N'' &= -  ( \frac{C'}{C} N)',\\
-s C' -  \ep C'' &= - C N.
\end{aligned} \end{align}
\begin{theorem}[\cite{Wa} Lemma 3.2, Lemma 3.4]\label{thm_NC}
A monotone 
 solution of \eqref{NC} for $\ep>0$ with the boundary conditions \eqref{bdry_nc} and \eqref{bdry_derivative_nc} exists if the relation $$n_- = (1+\ep)s^2$$ holds. More precisely  it holds that 
\begin{enumerate}
\item $N' <0$ and $C'>0$,
\item
$N(z)\sim e^{-sz}$, as $z\rightarrow \infty$,
 and
\item $\displaystyle \lim_{z\to -\infty} \frac{C'(z)}{C(z)} =  s$ and $\displaystyle \lim_{z\to \infty} \frac{C'(z)}{C(z)} = 0$.
\end{enumerate}
\end{theorem}
In \cite{Wa}, the author  used the results of the KPP-Fisher equation to establish the above theorem.
 \indent

The relation $ \calP = - C'/C$ gives the system
 \begin{align}\label{NP}\begin{aligned}
 -sN' -  N'' &=   (\calP
 N)',\\
-s \calP' -  \ep \calP'' &= (N-\ep\calP^2)'.
\end{aligned}\end{align}
We observe that $(N,P)=(N, (\calP,0))$ is a traveling wave solution of \eqref{nq}.
 From  \eqref{bdry_nc}  with the  above theorem, the wave $(N, \calP)$ is given the boundary condition 
 \begin{equation}\label{zeroth}
N(-\infty) =  (1+\ep)s^2,\, N(+\infty)=0,\, \calP(-\infty) = -s , \,\calP(+\infty)=0,\, N'(\pm\infty)=\PP'(\pm\infty)=0.
\end{equation}   We abbreviate $\displaystyle \lim_{z \to \pm \infty} f(z)$ by $f(\pm \infty)$ for any function $f$ on $\mathbb{R}$.
Moreover, the following theorem holds.
\begin{theorem}[\cite{LiWa} Theorem 2.1, or \cite{LiLiWa}]\label{thm_NP}
For $s>0$, if $\ep>0$ is small, then there exists a monotone 
solution  $(N,\calP)$ to \eqref{NP} with the boundary condition \eqref{zeroth}. In particular, it satisfies 
$0<N<
(1+\ep)s^2$ with
$N'  <0$ and $-s<\calP<0$ with $\calP'>0$, and   it is unique up to a translation.
\end{theorem}

The next lemma gives a uniform estimate of $N$ and $\PP$ for any small $\ep$. 
\begin{lemma}\label{lem_NP}
 
  For  $s>0$, there exist  constants   $\ep_1>0$ and  $L\in\mathbb{R}$ such that    if $(N,\calP)$ is a solution of \eqref{NP} and  \eqref{zeroth} for some $\ep\in(0,\ep_1)$
 given by Theorem \ref{thm_NP}, then 
  $$ |{N^{ (k)}}| \leq L, \quad   |{\calP^{(k)}}| \leq L,\quad \mbox{ for } 0\leq k\le 2,\quad \mbox{ and } $$
 $$
\Big|(\frac{1}{N})'\Big|+\Big|(\frac{1}{N})''
\Big|\leq \frac{L}{N},\quad 
\Big|(\frac{1}{\sqrt {N}})'\Big|\leq \frac{L}{\sqrt {N}}.$$
 
\end{lemma}
 \begin{proof}
 The estimate for $k\leq 1$ in the first line was proved   in \cite{LiLiWa} while the proof of the rest can be found in  Lemma $4.3$ in \cite{CCKL}.
\end{proof}

Lastly, we need the following lemma which gives a point in $\mathbb{R}$ contained both in the
   transition layer of $N$ and in that of  $\PP$.

\begin{lemma}\label{lem_center}
 
   For any $s>0$, there exists a constant $\ep_1>0$ such that if $(N,\calP)$ is a solution of \eqref{NP} and  \eqref{zeroth} for some $\ep\in(0,\ep_1)$
 given by Theorem \ref{thm_NP}, then there exists a point $z_0\in\bbR$ satisfying 
$$\calP(z_0)=-\frac{s}{2}\quad\mbox{ and }\quad  N(z_0)\geq \frac{s^2}{4}.$$
 
\end{lemma}

 \begin{proof}
 
  Since $\PP$ is continuous on $\bbR$ and $\PP(-\infty)=-s, \PP(+\infty)=0$, there    exists a point $z_0\in\bbR$ such that
$\calP(z_0)=-\frac{s}{2}$. To show $ N(z_0)\geq \frac{s^2}{4}$ for sufficiently small $\ep$, recall the equation \eqref{NP}. From 
$N(+\infty)=\PP(+\infty)=0$, we have 
\begin{align*} \begin{aligned}  
-s \calP -  \ep \calP' = (N-\ep\calP^2)
\end{aligned} \end{align*} 
 Assume that $\ep_1>0$ is smaller than $\ep_1$ in Lemma \ref{lem_NP}. Then for any $\ep\in(0,\ep_1)$, we get 
\begin{align*} \begin{aligned}  
N(z_0)=-s \calP(z_0)  +\ep((\calP(z_0))^2 -\calP'(z_0))
\geq \frac{s^2}{2}-\ep
|\calP'(z_0)|
\geq \frac{s^2}{2}-\ep_1L
\end{aligned} \end{align*} where $L$ is the constant from Lemma \ref{lem_NP}.
We take $\ep_1>0$ small enough to have 
$\ep_1
L
\leq \frac{s^2}{4}$.
Then $N(z_0)\geq \frac{s^2}{4}$ for any $\ep\in(0,\ep_1)$.
\end{proof}
 
 \begin{remark}\label{transition layer} The lemma is due to the fact 
 $N=-s\PP+\ep \mathcal{O}(1),$ which  means  the transition layers of  $N$ is overlapped with that of $\PP$  in some extent when $\ep$ is small enough.
\end{remark}
 
 \begin{remark} The first equation in  \eqref{NP} with \eqref{bdry_derivative_nc} and \eqref{zeroth} gives the simple relation between $N$ and $\calP$: 
 \begin{equation}\label{NP_relation}
 \frac{-N'}{N}=s+\calP.
 \end{equation}
 
  \end{remark}
   \begin{remark} 
If we denote $w(\cdot)=\frac{1}{N(\cdot)}$, then the above lemma implies
 \be\begin{split}\label{rem_center}
&\frac{w'(z)}{w(z)}\geq \frac{s}{2}\quad \mbox{for } z\geq z_0\quad \mbox{and}\quad w(z)\leq \frac{4}{s^2}\leq \frac{16}{s^4} N \quad \mbox{for } z\leq z_0.
 \end{split}\ee
 	Indeed, for $z\geq z_0$, we have 
 	$$\frac{w'}{w}=\frac{(1/N)'}{1/N}
=\frac{-N'/N^2}{1/N}=\frac{-N'}{N}=s+P\geq s+P(z_0)=\frac{s}{2}  
 	$$ thanks to \eqref{NP_relation} and $\calP'>0$. For $z\leq z_0$, we have 
 	$$
 	w=\frac{1}{N}\leq \frac{1}{N(z_0)}\leq \frac{4}{s^2}= \frac{16}{s^4}\cdot\frac{s^2}{4}\leq \frac{16}{s^4}N(z_0)\leq \frac{16}{s^4}N(z)  
 	$$ due to $N'<0$.
 We will use \eqref{rem_center} in the proof of Lemma \ref{lemma1_2}.
 \end{remark}
 
 Figure   $1$ describes the above discussions including
monotonicity of waves.
\begin{figure}[h]\label{figure_monotone}
\begin{center}
\includegraphics[scale=0.7]{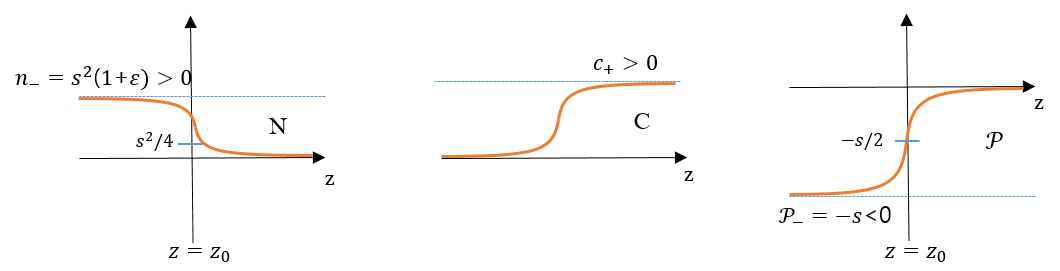}
\caption{ Monotonicity of $N$, $C$ and $\mathcal P$.}
\end{center}
\end{figure}

\end{subsection}

\begin{subsection}{Main theorem}\label{subsec_main_thm}
\indent
We recall \eqref{nq}:
 \begin{align}\label{np_main_eq} \begin{aligned}
\pa_t n - \Del n& = \na \cdot (np) \\
\pa_t p - \ep \Del p & =  -2\ep (p \cdot \na ) p + \na n,\qquad
 ( x, y,t) \in \bbr \times \mathbf{S}^\la\times  \bbr_+.
\end{aligned} \end{align}
Let   $(N,P)=(N,(\calP,0))$  be a traveling wave solution 
  of \eqref{np_main_eq}   with \eqref{zeroth}.  
In the below we introduce a weighted Sobolev space where our perturbative functions 
are constructed.
 We use the  weight function $w(\cdot_z)$ (only in the horizontal direction)  defined by
\[ w(z) = \frac{1}{N(z)}, \quad z\in \bbr\] where this unbounded weight was essentially
introduced in  \cite{LiLiWa} to handle the difficulty coming from the vacuum state $n_+=N(+\infty)=0$. 
 Note that $w$ is monotonically increasing, $w(-\infty)=\frac{1}{(1+\ep)s^2}$ and $w(z) \sim e^{sz}$ when $z\gg 1$ 
 by \eqref{zeroth} and Theorem $2.2$. \\
 \indent
For an integer $k\ge 0$ and for any $\la>0$, we define
   the Sobolev  spaces  $H^k$ and  a weighted Sobolev space  $H^k_{w}$ for functions periodic in $y$ with period $\la$  as follows;
\begin{align*}   H^k  & :=  \{ \vp\in {H_{loc}^k} (\mathbb{R}^2)\, |\, 
 \quad \sum_{i+j\le k} 
\sum_{n\in \mathbb{Z}}
 \int_\bbR  n^{2j}|\pa_z^i \varphi_n(z)|^2  dz < \infty,\, \vp(\cdot_z,\cdot_y+\la)=\vp(\cdot_z,\cdot_y) \},
\end{align*}
\begin{align*}   H^k_{w}  & :=   \{ \vp\in {H^k_{loc}} (\mathbb{R}^2))\, |\, 
 \quad \sum_{i+j\le k} 
\sum_{n\in \mathbb{Z}}
 \int_\bbR  n^{2j}|\pa_z^i \varphi_n(z)|^2 w(z) dz < \infty,\, \vp(\cdot_z,\cdot_y+\la)=\vp(\cdot_z,\cdot_y) \}
\end{align*}
 where  for each $n\in\mathbb{Z}$ and for each $z\in\mathbb{R}$, $ \varphi_n(z)$ is the $n$th Fourier coefficient of the ($\la$-) periodic (in $y$) function $\varphi(z, \cdot_y )$. \\
 We define the norms by \footnote{ The two quantities used to define $\| \cdot \|_{H^k}$ are equivalent up to the transversal length scale $\la>0$. In this paper, we do not pursue any  estimate which needs to hold uniformly on $\la$.}
\begin{align*}
&\| \varphi\|_{H^k}^2  := \sum_ { i+ j \le k}\int_{\bbr\times [0, \la]} | \pa_z^i \pa_y^j \varphi (z, y) |^2 dzdy, \quad\\
& \| \varphi\|_{H^k_w}^2 :=  \sum_ { i+ j \le k}\int_{\bbr\times [0, \la]} | \pa_z^i \pa_y^j \varphi (z, y) |^2 w(z) dzdy.
 \end{align*} 
  
Note that for any $f\in H_w^k$, we know
\be\label{norm_comp}
\| f\|_{H^k}^2\leq  {(1+\ep)s^2}\| f\|_{H_w^k}^2\ee
  due to $w\geq  \frac{1}{(1+\ep)s^2}$.\\


\indent We perturb the equation \eqref{np_main_eq} around the wave
  \begin{align}\label{np_perturb}
n(x, y, t) = N(x- st) + \dv\vp (x-st, y, t) \quad \mbox{and } 
\quad 
p(x, y, t) = P(x- st) + \na\psi (x-st, y, t).
\end{align}
With $z:=x-st$ in the  moving frame, we expect that for each time $t$,  the 
perturbation $(\vp(\cdot_z,\cdot_y,t), \psi(\cdot_z,\cdot_y,t))$ lies on  the following function class:
\begin{equation}\label{w_sobolev}
 \varphi=(\vp^1,\vp^2)\in (H^3_{w})^2 \quad\mbox{ and }\quad 
 \psi\in H^3\quad\mbox{ with } \quad \na \psi \in H^2_w.
\end{equation} 


\begin{remark}

Such a one-sided  decaying function  (in the weighted Sobolev space) appears typically  with respect to the solvability of $\dv \vp = u$ in the infinite cylinder  $\mathbb{R}\times\mathbf{S}^\la$ (\textit{e.g.} see \cite{Sol}). An explanation relevant to the perturbation \eqref{np_perturb} is given in Remark 1.2 in  \cite{CCKL}.

\end{remark} 
 
  Now we  state the main theorem:
 \begin{theorem}\label{theoremnc}  For any $s>0$ and for any $\la>0$ such that the product $s\cdot \la$ is sufficiently small, there exist constants $\ep_0>0$, $ K_0>0$,  and $ C_0\geq 1$  such that if $(N,\calP)$ is a solution of \eqref{NP} for some $\ep\in(0,\ep_0)$
with \eqref{zeroth} given by Theorem \ref{thm_NP}, then 
for any
 initial data $(n_0, p_0)$ of \eqref{np_main_eq} in the form of $  n_0 = N+\dv \varphi_0$ and $p_0=P+\nabla\psi_0  $ satisfying 
  $$M_0:=(\| \vp_0\|^2_{H^3_w} + \| \ps_0\|^2_{H^3}  + \|\nabla\psi_0\|^2_{H^2_w}  \color{black})\le K_0,$$  
 there exists a unique global  solution $(n,p)$  of  \eqref{np_main_eq}  in the form of
\[
n(x, y, t) = N(x- st) + \dv\varphi (x-st, y, t), 
\quad 
p(x, y, t) = P(x- st) + \nabla\psi (x-st, y, t),
\]
where $ \varphi|_{t=0} = \varphi_0$ and $\psi|_{t=0} = \psi_0$, and
  $ (\phi, \psi)$ satisfies the following inequality:
\be\label{main_est}   \sup_{t\in [0, \infty)}\Big(\| \varphi\|^2_{H^3_w} + \| \psi\|^2_{H^3} + \| \na \psi \|^2_{H^2_w}\Big)(t) 
+\int_0^{\infty}  \Big(  \| \na  \vp\|_{H^3_w}^2
+   \| \na  \ps\|_{H^2_w}^2
+\ep     \| \na^{4} \ps\|_w^2 \Big)(t)dt
 \le C_0M_0.\ee
 
\end{theorem} 
\begin{remark}\label{remark_nc}
 
 From \eqref{np_perturb}, \eqref{CP} and \eqref{CH}, 
 we have $c(\cdot_z +st)/C=e^{-\psi}$. Together with 
  $n(\cdot_z +st )-N=\dv\vp $, the above theorem implies
 \[   
 \sup_{t\in[0,\infty)}\Big(  \|n(\cdot_z +st,\cdot_y,t)-N(\cdot_z)  \|_{H^2_w}^2
+   \|  \nabla\big( 
\log c (\cdot_z +st,\cdot_y,t) - \log C(\cdot_z)
\big)   \|_{H^2_w}^2
\Big)
\]
 \[   
+\int_0^\infty  \Big(  \|n(\cdot_z +st,\cdot_y,t)-N(\cdot_z)  \|_{H^3_w}^2
+   \|  \nabla\big( 
\log c (\cdot_z +st,\cdot_y,t) - \log C(\cdot_z)
\big)   \|_{H^2_w}^2
\Big)dt\leq C_0M_0.\]
  
\end{remark}

Before closing this subsection we give a summary on notations used in the paper.
 \[\Om = \bbr \times [0, \la],\]
\[ w(z) : =\frac{1}{N(z)},\]
\[M(t):=  \sup_{ s\in [0, t]}( \| \varphi(s)\|^2_{H^3_w} 
+ \|\psi(s)\|^2_{H^3} +\|\na \psi(s)\|^2_{H^2_w}),\]
\[ M_0:=  ( \| \varphi_0\|^2_{H^3_w} 
+ \|\psi_0\|^2_{H^3}+\|\na \psi_0\|^2_{H^2_w})  \]
\[  \|  f\|:=\|  f\|_{L^2(\Om)},\]
\[  \|  f\|^2_{k}:=\|  f\|^2_{H^k}=\sum_{|\alpha|=0}^k\int_\Omega|D^\alpha f|^2\,dzdy,\]
\[  \|  f\|^2_{k,w}:=\|  f\|^2_{H^k_w}=\sum_{|\alpha|=0}^k\int_\Omega|D^\alpha f(z,y)|^2\,w(z)dzdy,
\]
\[\int f:=\int_\Omega f(z,y) dzdy,\]
\[\int_0^t g:=\int_0^t g(\sigma) d\sigma.\]

Here we use the notation $\| \cdot \|$ to indicate certain norm in space only. For instance, when $f$ is time-dependent then $\|f\|$ means $\|f(t)\|$ in the sequel.

\end{subsection}

\begin{subsection}{Perturbation equation}\label{hopf}

In this subsection, we derive the system on $(\vp, \psi)$ first. Next we state the main propositions including results on the local existence and the uniform estimates  of  $(\vp, \psi)$.

From \eqref{np_main_eq} and \eqref{np_perturb}, by setting 
 \begin{align} \label{uv}
 u = \na \cdot \varphi\quad  \mbox{ and } \quad  v = \na \psi 
 \end{align}
  temporarily in \eqref{np_perturb}, we obtain 
\begin{align}\label{perturb} \begin{aligned}
&u_t - s u_z - \Del u
=  \na \cdot ( Np + Pu+ up ),\\
&v_t - s v_z - \ep\Del v =  -2\ep \left(  ( (P+v )\cdot \na ) (P+v) - (P\cdot \na) P \right) + \na  u.
\end{aligned}\end{align}
Plugging the relation \eqref{uv} in \eqref{perturb} and taking off derivatives, we find that the antiderivative $(\vp, \psi)=((\vp^1,\vp^2),\psi)$ of $(u,v)$ 
satisfies the system
\begin{align}\label{main_eq} \begin{aligned}
\varphi_t - s \varphi_z - \Del \varphi
& =   N \na \psi + P\na \cdot \varphi + \na\cdot \varphi \na \psi, \\
\ps_t - s \ps_z -\ep \Del \ps &=  -2\ep  P \cdot\na \ps -\ep  |\na \ps|^2 + \dv\varphi
\end{aligned}
\end{align}
for $
 ( z, y,t) \in \bbr \times \mathbf{S}^\la\times  \bbr_+
 $. In doing so, we use the curl free property of $v$ and $p$. Here the term $P\dv\vp$ means the vector $
 (\calP\dv\vp,0)$.  The multidimensional setting \eqref{uv}
 was proposed in \cite{CCKL}. Looking for a perturbation in our system as an antiderivative follows the setting in one dimensional 
 works \cite{JinLiWa}, \cite{LiLiWa}. This method 
   can be  found in the study on the nonlinear stability of  shock profiles of viscous conservation laws
  under the mean zero condition with a weight function since the papers \cite{KaMa} and \cite{Go}. Without the mean zero condition, we refer to \cite{Liu}, \cite{Zou}, \cite{smol} and references therein. \\

\indent 
We obtain Theorem \ref{theoremnc} immediately without any difficulty once we prove the proposition below.
\begin{proposition}\label{zeroglobal}  

For any $s>0$ and any $\la>0$ such that the product $(s\cdot\la)$ is sufficiently small, there exist constants $\ep_0>0$, $K_0>0$ and $C_0\geq 1$ such that if $(N,\calP)$ is a solution of \eqref{NP} for some $\ep\in(0,\ep_0)$
with \eqref{zeroth} given by Theorem \ref{thm_NP}, then 
  we have the following:\\

For any
 initial data $(\vp_0, \psi_0)$ of \eqref{main_eq}   satisfying 
 $$M_0:=\| \vp_0\|^2_{H^3_w} + \| \ps_0\|^2_{H^3}  + \|\nabla\psi_0\|^2_{H^2_w}  \color{black}\le K_0,$$  
 there exists a unique global  solution $(\vp,\psi)$  of  \eqref{main_eq}   
where $ \varphi|_{t=0} = \varphi_0$ and $\psi|_{t=0} = \psi_0$, and  $ (\phi, \psi)$ satisfies the following inequality:
\[   \sup_{t\in [0, \infty)}\Big(\| \varphi\|^2_{H^3_w} + \| \psi\|^2_{H^3} + \| \na \psi \|^2_{H^2_w}\Big) 
+\int_0^{\infty}  \Big(  \| \na  \vp\|_{H^3_w}^2
+   \| \na  \ps\|_{H^2_w}^2
+\ep     \| \na^{4} \ps\|_w^2 \Big)dt
 \le C_0M_0.\]

\end{proposition}
Proposition \ref{zeroglobal} is a consequence of the following two propositions:  Proposition \ref{zerolocal} which gives a local-in-time existence result and Proposition \ref{uniform_} which shows an \textit{a priori} uniform-in-time estimate.

\begin{proposition}\label{zerolocal}
Let $s>0$ and $\la>0$. For sufficiently small $\epsilon>0$,
if $(N,\calP)$ is a solution of \eqref{NP} 
with \eqref{zeroth} given by Theorem \ref{thm_NP}, then  for any 
$M>0$,  there exists $T_0>0$ such that 
for any data $ (\varphi_0, \psi_0)$ with $ \| \varphi_0\|^2_{H^3_w} + \| \psi_0\|^2_{H^3} + \|\na \psi_0\|^2_{H^2_w}  \leq M $,
 the system \eqref{main_eq} has
a unique solution $(\vp, \psi)$ on $[0, T_0]$ satisfying 
\[  \varphi \in L^{\infty}(0, T_0; H^3_w), \, \psi \in L^{\infty}(0, T_0; H^3), \, \nabla\psi \in L^{\infty}(0, T_0; H^2_w)  \color{black} \mbox{ with }  \varphi|_{t=0} = \varphi_0,\, \psi|_{t=0} = \psi_0\] and
\begin{equation*}
 \sup_{t\in [0, T_0]}\Big( \| \varphi\|^2_{H^3_w}+ \| \psi \|^2_{H^3_w }+\| \na \psi\|^2_{H^2_w}\Big) \le 2 M.
 \end{equation*}
\end{proposition}
\indent
The local  solution of \eqref{main_eq} can be obtained 
 by the usual contraction method and by  a similar computation as in the proof of Proposition \ref{uniform_}, for which we omit its proof (or see \cite{CCKL}).\\  
 
 \noindent
The following proposition gives a uniform-in-time estimate, which is the main heart of this paper.
\begin{proposition}\label{uniform_} 
For any $s>0$ and any $\la>0$ such that the product $(s\cdot\la)$ is sufficiently small, there exist constants $\ep_0>0$, $\delta_0>0$ and $C_0\geq 1$ such that if $(N,\calP)$ is a solution of \eqref{NP} for some $\ep\in(0,\ep_0)$
with \eqref{zeroth} given by Theorem \ref{thm_NP}, then 
  we have the following:\\
 If  $(\varphi, \psi)$ be a local solution of \eqref{main_eq} on $[0,T]$ for some $T>0$ with $M(T)\leq\delta_0$, 
then   we have 
\[   M(T) +  \int_0^{T}  \sum_{l = 1}^4 \| \na^{l} \vp\|_w^2
+\int_0^{T}   \sum_{l = 1}^3\| \na^l \ps\|_w^2
+\ep\int_0^{T}     \| \na^{4} \ps\|_w^2 \leq C_0 M(0).\]

\end{proposition}
Note that $C_0$ does not depend on $T>0$.\\

\begin{proof}[Proof of Proposition \ref{zeroglobal} from Proposition \ref{zerolocal} and Proposition \ref{uniform_}] We include the proof here for readers' convenience even if  this continuation argument is now standard (or see \cite{CCKL}).
Let's take $M:=\delta_0/2$ and $K_0:=M/C_0$ where $\delta_0>0$ and $C_0\geq 1$ are the constants in  Proposition \ref{uniform_}. Due to $C_0\geq 1$, we know $K_0\leq M$.
Consider the initial data $(\vp_0, \ps_0)$ with $
M_0\leq K_0
$. By using the constant $M$ to the local-existence result (Proposition \ref{zerolocal}),
  there exist  $T_0>0$,
and there is the  unique local solution $(\vp, \ps)$ on $[0,T_0]$ with $M(T_0)\leq 2M$. Due to
$M(T_0)\leq2M\leq \delta_0$, we can use the result of Proposition \ref{uniform_} to obtain $M(T_0)\leq C_0 M(0)= C_0 M_0$, which implies $M(T_0)\leq C_0 K_0\leq M$. Hence we can extend the solution from the time $T_0$  up to the time $2T_0$ by Proposition \ref{zerolocal} and we obtain
  $M(2T_0) \leq 2M\leq \delta_0$. Again by Proposition \ref{uniform_}, it implies $M(2T_0)\leq C_0M_0\leq M$. Thus we can repeat this process of the extension to get
  $M(kT_0)\leq C_0M_0$ for any $k\in\mathbb{N}$. 
\end{proof}

  In the rest of the paper, we focus on proving Proposition \ref{uniform_}.
 
\end{subsection}

\section{Uniform-in-time estimate: Proof of Proposition \ref{uniform_}}\label{sec_three_main}

Let $s>0$ and $\lambda>0$. 
Recall  the
Poincar\'{e} inequality on  intervals which says that there is a constant $C_{p}>0$ such that for any $\la>0$ and 
 for any $f\in W^{1,2}(0,\la)$, the inequality \begin{equation}\label{ineq:poincare2} 
\| f- \overline{f}\|_{L^2(0,\la)} \le \la  C_p \|f'\|_{L^2(0,\la)}  
\end{equation} holds. Here  the mean value $\overline{f}$ of $f$  is defined by $\overline{f}:=\frac 1 \la\int_0^\la f(y)dy$.
We assume that the product $s\cdot\la$ is small to have 
\begin{equation}\label{smallness_lemma0_-_la}
s\cdot\la\cdot C_{p} \leq \frac{1}{16}.
\end{equation} 
 From now on, these values $s>0$ and $\la>0$ are fixed until the end of the proof. 
Let's assume  $0<\ep_0\leq 1$  and $0<\delta_0\leq 1$ which will be taken sufficiently small later in the proof several times. \\

\indent We suppose first that  $\ep_0>0$ is sufficiently small so that any $\ep\in(0,\ep_0]$  meets the assumption of Theorem \ref{thm_NP}.
Let $(N,\calP)$ be a solution of \eqref{NP}
for some $\ep\in(0,\ep_0]$ with \eqref{zeroth} given by Theorem \ref{thm_NP}.
Let $(\varphi, \psi)$ be a local solution of \eqref{main_eq} on $[0,T]$ for some $T>0$ with $M(T)\leq\delta_0$.\\

\indent In the sequel,    
 $C$ denotes a positive constant which may change
from line to line, but which stays independent on ANY choice of $\ep\in(0,\ep_0)$ and $T>0$ as long as the positive parameters $\ep_0$ and $\delta_0$ are sufficiently small.\\
\subsection{Zero-th order estimate}\label{sec_zero}
\ \\
\begin{lemma}\label{lemma0_-}
 If the positive constants $\ep_0,  \delta_0 $ are sufficiently small, then
 there exists a constant 
  $C_1>0 $ such that 
 for any $t\in[0,T]$,

\begin{equation}\begin{split} \label{eq_lemma0_-}
&\| \psi \|^2 + \|{\vp}\|_w^2 + \int_0^t \|{\na \varphi}\|_w^2
+{\ep}\int_0^t \|\na \psi\|^2
+  \int_0^t\int\Big( \frac{(N')^2}{N^3}   {|\vp|^2} 
 +   \fn{\calP'}   {(\vp^1)^2} 
  +   \frac{ \calP N'}{N^2}   {(\vp^2)^2} \Big)
\\
&  \le C_1 \cdot( \| \psi_0\|^2 + \| {\varphi_0}\|_w^2 )
 +\ep\cdot  C_1\int_0^t \int\calP'|\psi|^2  
+ C_1\cdot{M(t)}\cdot \int_0^t 
 \|{\na \psi}\|_w^2.
\end{split}\end{equation}

\end{lemma}
 \begin{remark}\label{rem_difficulty} 
We note that the term $$\int_0^t\int\Big( \frac{(N')^2}{N^3}   {|\vp|^2} 
 +   \fn{\calP'}   {(\vp^1)^2} 
  +   \frac{ \calP N'}{N^2}   {(\vp^2)^2} \Big)$$ in the left-hand side of 
  \eqref{eq_lemma0_-} plays a role of dissipation on the zero-th order. This is  non-symmetric for $\vp_1$ and $\vp_2$ due to the non-symmetric structure of the main equation \eqref{main_eq}. In Lemma \ref{lemma0_temp}, these localized $L^2$-norms of $\vp$ will be used to control  $$\ep\int_0^t \int\calP'|\psi|^2$$   in the right-hand side of  \eqref{eq_lemma0_-}, which is 
a localized $L^2$-norm of $\psi$ multiplied by $\ep$.

 \end{remark}
\begin{proof}

We multiply  $\fn{\varphi}$ to the $\varphi$ equation and $\psi$ to the $\psi$ equation:

\ben\begin{split}
& \fn{1}\vp\cdot(  \varphi_t - s \varphi_z - \Del \varphi)
+ \psi( \ps_t - s \ps_z -\ep \Del \ps)\\
&
 =  \fn{1}\vp \cdot(N \na \psi + P\na \cdot \varphi + \na\cdot \varphi \na \psi )+ \psi(  -2\ep  P \cdot\na \ps -\ep  |\na \ps|^2 + \dv\varphi).
\end{split}\een  Thus we get 
\ben\begin{split}
& 
(\fn{1}|\vp|^2/2)_t - s(\fn{1}|\vp|^2/2)_z + s(\fn{1})'|\vp|^2/2+
\fn{1}\vp\cdot(  - \Del \varphi)\\
&\quad\quad +(|\psi|^2/2)_t
-s(|\psi|^2/2)_z
+\psi(   -\ep \Del \ps)\\
&
 =  
  \vp \cdot \na \psi+ 
  \fn{\PP}\vp^1\na \cdot \varphi
+ \fn{1}(\na\cdot \varphi) \vp \cdot  \na \psi 
  -2\ep\PP\psi \psi_z 
-\ep \psi |\na \ps|^2 
+\psi\dv\vp.
\end{split}\een

By integrating in space $\Omega$, we have
\begin{align*}
&\frac 12 \ddt \left( \int \fn{ |\varphi|^2} +\int  |\psi|^2 \right) 
+ \int\fn{   |\na \varphi |^2} + \ep\int |\na \psi|^2+ 
\frac s 2 \int |\varphi|^2 \left( \fn{1}\right)'\\
&=  -  \int \vp\cdot\vp_z \left(  \fn{1}\right)'+
\int \fn{\calP} \varphi^1 \na \cdot \varphi  
+  \int \fn{\varphi \cdot \na\psi} \na\cdot \varphi
-
 2\ep \int \calP \ps_z \psi
- \ep \int |\na \ps|^2 \ps.
\end{align*} 
Here we use the notation
$$   |\na \varphi |^2:= \sum_{i=1}^2 |\na \varphi^i|^2.$$


Recall the Sobolev embedding which gives us a constant $C_{SV}>0$ such that for any $f\in H^2(\Omega)$, the inequality 
$$\|f\|_{L^\infty}\leq C_{SV} \|f\|_{H^2}$$ holds.
 We control the cubic term:
\begin{equation}\begin{split}\label{cubic_0}
  \int \Big| \fn{\varphi \cdot \na\psi} \na\cdot \varphi\Big| & \le \Big( 4\| \vp\|_{L^{\infty}} ^2\int \fn{| \na \psi|^2} 
 + \frac{1}{16} \int \fn{|\na \vp|^2}\Big) \\ 
 & \le C\cdot {M(t)} \int \fn{| \na \psi|^2}  + \frac{1}{16}\int \fn{|\na \vp|^2}
\end{split}\end{equation}
 where we used $\|\vp\|_{L^\infty}\leq C_{SV} \|\vp\|_{H^2}
\leq C_{SV}\cdot s\sqrt{1+\ep_0} \|\vp\|_{H_w^2} 
  \leq C \sqrt{M(t)}$ due to \eqref{norm_comp} and $0<\ep\leq\ep_0\leq 1$.\\ 

  We control the quadratic term:
  \be\begin{split}\label{quad_control_0-}
  \int \fn{\calP} \varphi^1\na \cdot \varphi   &  =
  \int \fn{\calP} (\frac{(\vp^1)^2}{2})_z
  +\int  \fn{\calP}   \vp^1 (\vp^2)_y \\
   &  =
 - \int (\fn{\calP})'  \frac{(\vp^1)^2}{2} 
  +\int_\bbR \fn{\calP}(z)\int_0^\la (\vp^1(z,y)-\overline{\vp^1}(z))(\vp^2)_y(z,y)\,dydz\\
\end{split}\ee  where $\overline{\vp^1}(z):=\frac 1\la \int_0^\la \vp^1(z,y)dy$.
Note $\Big|\fn{\PP}\Big|\leq \frac{s}{N}$ and the
Poincar\'{e} inequality \eqref{ineq:poincare2} on an interval $(0,\la)$.
Thus we get
  \be\begin{split}\label{use_poincare}
  \Big|\int_\bbR \fn{\calP}(z)\int_0^\la (\vp^1(z,y)-\overline{\vp^1}(z))(\vp^2)_y(z,y)\,dydz \Big| &\leq sC_p\la\int_\bbR\frac{\|(\vp^1)_y(z)\|_{L^2(0,\la)} {\|(\vp^2)_y(z)\|_{L^2(0,\la)}}}{ {N}} dz \\
   &\leq sC_p\la\int \fn{|\nabla\vp|^2}\leq\frac{1}{16}\int \fn{|\nabla\vp|^2} 
\end{split}\ee where we used the assumption \eqref{smallness_lemma0_-_la}. 
 Then we have
 \begin{align*}
  \int \fn{\calP} \varphi^1\na \cdot \varphi    
   &  \leq 
 - \int (\fn{\calP})'  \frac{(\vp^1)^2}{2} 
  +\frac{1}{16}\int \fn{|\nabla\vp|^2}.
\end{align*}  

For the $\ep$-terms, we assume that $\delta_0>0$ is  small enough to get 
\begin{equation*}
C_{SV}s\sqrt{2}\sqrt{\delta_0}\leq 1/4.
\end{equation*}
Then, we have
\begin{align*}
 -2\ep \int \calP \ps_z \psi &
=\ep\int\calP'|\psi|^2\quad\mbox{and} \\
 - \ep \int |\na \ps|^2 \ps&
 \leq \frac{1}{4}\ep   \|\na\psi\|^2
\end{align*}
where we used $\|\psi\|_{L^\infty}\le C_{SV}\|\psi\|_{H^2}\leq C_{SV}s\sqrt{2}\|\psi\|_{H^2_w}
\leq C_{SV}s\sqrt{2}\sqrt{\delta_0}\leq \frac 1 4$.

Up to now, we have
\begin{align*}
&\frac 12 \ddt \left( \int \fn{ |\varphi|^2} +\int  |\psi|^2 \right) 
+ \frac{7}{8}\int\fn{   |\na \varphi |^2} + \frac{3}{4}\ep\int |\na \psi|^2 
\\
&\leq  \underbrace{ -  \int \vp\cdot\vp_z \left(  \fn{1}\right)'
-\frac s 2 \int |\varphi|^2 \left( \fn{1}\right)'
 - \int (\fn{\calP})'  \frac{(\vp^1)^2}{2}  }_{=:(*)} 
+\ep\int\calP'|\psi|^2  
 + C\cdot    {M(t)} \int \fn{| \na \psi|^2}.  
\end{align*}
We observe 
\begin{align*}
&  (*) =
 \int \vp\cdot\vp_z \frac{N'}{N^2}
  - \int (\fn{s+\calP})'  \frac{|\vp|^2}{2}
  +  \int (\fn{\calP})'  \frac{(\vp^2)^2}{2} 
  \\ &  \leq  
     \int |\vp||\vp_z |\frac{N'}{N^2}
  - \int \fn{\calP'}  \frac{|\vp|^2}{2}
   - \int \frac{(s+\calP   )(-N')}{N^2}  \frac{|\vp|^2}{2}
  +  \int \fn{\calP'}   \frac{(\vp^2)^2}{2}
  -   {\int \frac{ \calP N'}{N^2}  \frac{(\vp^2)^2}{2}}  
    \\ &  \leq 
      \frac{3}{4} \int \fn{|\vp_z|^2}
  +  \frac 13 \int |\vp|^2\frac{(N')^2}{N^3}       
   -  \int \frac{(N')^2}{N^3}  \frac{|\vp|^2}{2}
 - \int \fn{\calP'}   \frac{(\vp^1)^2}{2}
 - \int \frac{ \calP N'}{N^2}  \frac{(\vp^2)^2}{2} 
    \\ &  \leq 
      \frac{3}{4} \int \fn{|\na\vp|^2}
    -\Big(\frac 16 \int \frac{(N')^2}{N^3}  \frac{|\vp|^2}{2}
 + \int \fn{\calP'}   \frac{(\vp^1)^2}{2}
  + \int \frac{ \calP N'}{N^2}  \frac{(\vp^2)^2}{2} \Big).
\end{align*}
Thanks to Theorem \ref{thm_NP}, we observe $$\frac{(N')^2}{N^3}  > 0, \,\fn{\calP'}  > 0\quad \,\mbox{and }\quad \frac{ \calP N'}{N^2}> 0.$$ Thus we get 
 \begin{equation*}\begin{split}
&\frac 12 \ddt \left( \int \fn{ |\varphi|^2} +\int  |\psi|^2 \right) 
+ \frac{1}{8}\int\fn{   |\na \varphi |^2} + \frac{3}{4}\ep\int |\na \psi|^2 \\&\quad\quad +\frac 16 \int \frac{(N')^2}{N^3}  \frac{|\vp|^2}{2}
 + \int \fn{\calP'}   \frac{(\vp^1)^2}{2}
  + \int \frac{ \calP N'}{N^2}  \frac{(\vp^2)^2}{2}
\\
&\leq   \ep\int\calP'|\psi|^2  
  +C\cdot   {M(t)} \int \fn{| \na \psi|^2}. 
\end{split}\end{equation*} 
Integrating in time gives the lemma.
\end{proof}

\begin{lemma}\label{lemma0_temp}

 If the positive constants $\ep_0,  \delta_0 $ are sufficiently small, then there exists a constant 
  $C_2
>0$ such that 
 for any $t\in[0,T]$,

 \ben
 \int_0^t\int \PP'|\psi|^2
 \leq
C_2\cdot [\mbox{LHS of } \eqref{eq_lemma0_-}  ]
   + 
   C_2\cdot 
     M(t) \int_0^t\int \fn{| \na \psi|^2}.   
\een
Here LHS is shorthand for the left-hand side.

\end{lemma}

\begin{proof}

We multiply  $\fn{\PP}\vp$ to the $\varphi$ equation and $\PP\psi$ to the $\psi$ equation: 
\ben\begin{split}
& \fn{\PP}\vp\cdot(  \varphi_t - s \varphi_z - \Del \varphi)
+\PP\psi( \ps_t - s \ps_z -\ep \Del \ps)\\
&
 =  \fn{\PP}\vp \cdot(N \na \psi + P\na \cdot \varphi + \na\cdot \varphi \na \psi )+\PP\psi(  -2\ep  P \cdot\na \ps -\ep  |\na \ps|^2 + \dv\varphi).
\end{split}\een Thus we get  
\ben\begin{split}
& 
(\fn{\PP}|\vp|^2/2)_t - s(\fn{\PP}|\vp|^2/2)_z + s(\fn{\PP})'|\vp|^2/2+
\fn{\PP}\vp\cdot(  - \Del \varphi)\\
& +(\PP|\psi|^2/2)_t
-s(\PP|\psi|^2/2)_z+ s\PP'|\psi|^2/2
+\PP\psi(   -\ep \Del \ps)\\
&
 =  
  {\PP}\vp \cdot \na \psi+ 
  \fn{\PP^2}\vp^1\na \cdot \varphi
+ \fn{\PP}(\na\cdot \varphi) \vp \cdot  \na \psi   -2\ep\PP^2\psi \psi_z 
-\ep \PP\psi |\na \ps|^2 
+\PP\psi\dv\vp.
\end{split}\een
Then we integrate on $\Omega$ to get
\ben\begin{split}
& 
\frac{d}{dt}\int(\fn{\PP}|\vp|^2/2+ \PP|\psi|^2/2) 
 +s\int \PP'|\psi|^2/2
\\
&
 = 
-\int(\fn{\PP})'\vp_z\cdot\vp - s\int(\fn{\PP})'|\vp|^2/2
   -\ep\int\PP'\psi\psi_z
- \int\fn{\PP}|\na\vp|^2     
      -\ep\int\PP|\na\psi|^2\\
   &
-\int \PP'\psi\vp^1
 +\int \fn{\PP^2}\vp^1\na \cdot \varphi
+ \int\fn{\PP}(\na\cdot \varphi) \vp \cdot  \na \psi  
 -2\ep\int\PP^2\psi \psi_z 
-\ep \int\PP\psi |\na \ps|^2.
\end{split}\een

For the cubic term  as in \eqref{cubic_0}  with $|\PP|\leq s$,   we have
\begin{align*}
\Big| \int\fn{\PP}(\na\cdot \varphi) \vp \cdot  \na \psi \Big|\leq   s\int \Big| \fn{\varphi \cdot \na\psi} \na\cdot \varphi\Big| 
& \le  C\cdot  {M(t)} \int \fn{| \na \psi|^2}  +C\int \fn{|\na \vp|^2}.
\end{align*}

For the quadratic term, we get
\ben\begin{split}
&  
 - \int\fn{\PP}|\na\vp|^2  
 \leq C \int \fn{|\na\vp|^2}\quad \mbox{ and}
\end{split}\een
\ben\begin{split}
&  
-\int \PP'\psi\vp^1\leq \frac{s}{8}\int \PP'|\psi|^2
+\frac{2}{s}\int \PP'|\vp^1|^2
\leq \frac{s}{8}\int \PP'|\psi|^2
+C\int \fn{\PP'}|\vp^1|^2.
\end{split}\een
As we did in and after \eqref{quad_control_0-},  
\ben\begin{split}
&  
\int \fn{\PP^2}\vp^1\na \cdot \varphi 
\leq 
 - \int (\fn{\calP^2})'  \frac{(\vp^1)^2}{2} 
    +C\int \fn{|\nabla\vp|^2}\\
\end{split}\een by using  \eqref{smallness_lemma0_-_la}.


For $\ep-$terms, we assume that  $\ep_0>0$ is smaller than $\ep_1$ in Lemma 
\ref{lem_NP}. Then
we estimate
\ben\begin{split}
&  
  -2\ep\int\PP^2\psi \psi_z =\ep\int(\PP^2)'|\psi|^2=
  2\ep\int\PP \PP'|\psi|^2\leq 0,
\end{split}\een
\ben\begin{split}
&  
  -\ep\int\PP|\na\psi|^2
  \leq C\ep   \int |\na\psi|^2,
\end{split}\een
\ben\begin{split}
&  
 -\ep \int\PP\psi |\na \ps|^2 \leq  \ep s \int|\psi| |\na \ps|^2
 \leq  C_{SV}\cdot \ep s \sqrt{M(t)} \int |\na \ps|^2
   \leq  C \cdot\ep   
  \int |\na \ps|^2,\quad \mbox{ and}
\end{split}\een
\ben\begin{split}
 -\ep\int\PP'\psi\psi_z&\leq \frac{s}{4}\int \PP'|\psi|^2+\frac{\ep^2}{s}\int \PP'|\psi_z|^2
 \leq \frac{s}{4}\int \PP'|\psi|^2+\frac{\ep^2L }{s}\int |\na\psi|^2\\ &
  \leq \frac{s}{4}\int \PP'|\psi|^2+C\cdot\ep\int |\na\psi|^2
\end{split}\een  where we used $\delta_0\leq1$, $\ep_0\leq1$,
and $|\calP'|\leq L$ where $L$ is the constant 
 in Lemma \ref{lem_NP}. 

Up to now, we have
\ben\begin{split}
& 
\frac{d}{dt}\int(\fn{\PP}|\vp|^2/2+ \PP|\psi|^2/2) 
 +\frac{s}{8}\int \PP'|\psi|^2
\\
&
 \leq  
\underbrace{ - s\int(\fn{\PP})'|\vp|^2/2
 - \int (\fn{\calP^2})'  \frac{(\vp^1)^2}{2}-\int(\fn{\PP})'\vp_z\cdot\vp }_{:=(**)}
  \\
&\quad
  +C\cdot \ep
  \int |\na\psi|^2
 + C
 \int\fn{ |\na\vp|^2}\\ 
   & \quad
+ C
\int \fn{\PP'}|\vp^1|^2
 +   
C\cdot 
 M(t) \int \fn{| \na \psi|^2}.  
\end{split}\een

For the first two terms in $(**)$, we have 
\ben\begin{split}    
&  - s\int(\fn{\PP})'|\vp|^2/2
 - \int (\fn{\calP^2})'  \frac{(\vp^1)^2}{2} 
\\ &
=-s\int\fn{\PP'}|\vp|^2/2
+s\int \frac{\PP N'}{N^2}|\vp|^2/2
-\int \fn{2\PP\PP'} \frac{(\vp^1)^2}{2}
-\int \PP^2(\fn{1})' \frac{(\vp^1)^2}{2}
 \\ &
 \leq {-s\int\fn{\PP'}(\vp^2)^2/2} +\frac{s}{2}\int \fn{\PP'} {(\vp^1)^2}
 +s\int \frac{\PP N'}{N^2}(\vp^2)^2/2
 +\int \frac{(s+\PP)\PP N'}{N^2}(\vp^1)^2/2
  \\ &
\leq   -s\int\fn{\PP'}(\vp^2)^2/2   +C\int \fn{\PP'} {(\vp^1)^2}
 +C\int \frac{\PP N'}{N^2}(\vp^2)^2/2
+C\int \frac{ (N')^2}{N^3}(\vp^1)^2/2.
\end{split}\een 

For the last term in $(**)$, we have
\ben\begin{split} 
 &-\int(\fn{\PP})'\vp_z\cdot\vp  =
 -\int(\fn{\PP'}-\frac{\PP N'}{N^2})\vp_z\cdot\vp \\ &
 =-\int\fn{\PP'}(\vp^1)_z\cdot\vp^1
 -\int\fn{\PP'}(\vp^2)_z\cdot\vp^2
 +\int \frac{\PP N'}{N^2}\vp_z\cdot\vp 
 \\ &
 \leq (\frac{1}{2}+\frac{1}{s})\int\fn{\PP'}|\vp_z|^2 
 + \frac 12\int\fn{\PP'}(\vp^1)^2
+ \frac {s}{4}\int\fn{\PP'}(\vp^2)^2
 +  \frac 12 \int \fn{|\PP|^2|\vp_z|^2} + \frac 12 \int \frac{(N')^2}{N^3}|\vp|^2
  \\ &
 \leq ((\frac{1}{2}+\frac{1}{s})L+\frac{L^2}{2})\int\fn{|\na \vp|^2} + \frac 12\int\fn{\PP'}(\vp^1)^2
  + \frac 12 \int \frac{(N')^2}{N^3}|\vp|^2
+ \frac s4\int\fn{\PP'}(\vp^2)^2
 \\ &
 \leq C\int\fn{|\na \vp|^2} + C\int\fn{\PP'}(\vp^1)^2
  +   C \int \frac{(N')^2}{N^3}|\vp|^2
+ \frac s4\int\fn{\PP'}(\vp^2)^2.
\end{split}\een 

We combine the above two computations to get
\ben\begin{split} 
  (**) &
 \leq 
\underbrace{-\frac{s}{4}\int\fn{\PP'}(\vp^2)^2/2}_{\leq0} 
    +C\int\fn{|\na \vp|^2}  
 +  C\int\fn{\PP'}(\vp^1)^2
  + C \int \frac{(N')^2}{N^3}|\vp|^2
   +C\int \frac{\PP N'}{N^2}(\vp^2)^2.
\end{split}\een

In sum, we have
\ben\begin{split}
& 
 \frac{s}{8}\int \PP'|\psi|^2
 \leq  -\frac{d}{dt}\int(\fn{\PP}|\vp|^2/2+ \PP|\psi|^2/2) 
 \underbrace{-\frac{s}{4}\int\fn{\PP'}(\vp^2)^2/2}_{\leq0} \\
 &\quad
 +C
 \int\fn{|\na \vp|^2}  
 + C
 \int\fn{\PP'}(\vp^1)^2
  + C
   \int \frac{(N')^2}{N^3}|\vp|^2
   +C
   \int \frac{\PP N'}{N^2}(\vp^2)^2
  \\
&\quad
  +C\cdot \ep
  \int |\na\psi|^2
 + C\cdot 
 M(t) \int \fn{| \na \psi|^2}.  
\end{split}\een 
By taking integral in time, we have the lemma
since $\PP<0$ and $|\PP|\leq s$ implies 
\ben\begin{split}
-\int_0^t\frac{d}{dt}\int(\fn{\PP}|\vp|^2/2+ \PP|\psi|^2/2) 
&=
\underbrace{\int(\fn{\PP}|\vp_0|^2/2+ \PP|\psi_0|^2/2)}_{\leq 0} 
-\int(\fn{\PP}|\vp(t)|^2/2+ \PP|\psi(t)|^2/2)\\ &
\leq C
(\|\vp(t)\|_w^2+\|\psi(t)\|^2).
\end{split}\een 

\end{proof}
We combine Lemma \ref{lemma0_-} with Lemma \ref{lemma0_temp}
   in the following way:
We assume  $\ep_0>0$ small enough to have
\ben
\ep_0C_1C_2\leq \frac{1}{2}
\een where $C_1$ is from Lemma \ref{lemma0_-} and $C_2$ is from Lemma \ref{lemma0_temp}. Then add [$\ep C_1\cdot$(the resulting estimate of Lemma \ref{lemma0_temp})] to \eqref{eq_lemma0_-} 
to get 
\ben 
\frac{1}{2}\cdot(\mbox{LHS of } \eqref{eq_lemma0_-}  )
\leq C_1 \cdot( \| \psi_0\|^2 + \|  {\varphi_0}\|_w^2 )
+ (C_1+ \ep_0 C_1  
 C_2
 )
\cdot{M(t)}\cdot \int_0^t 
 \| {\na \psi}\|_w^2. \een 
In sum, we have the following zero-th order estimate which hasn't been closed yet:
\begin{equation}\label{eq_lemma0_}
\| \psi \|^2 + \|{\vp}\|_w^2 + \int_0^t \| {\na \varphi}\|_w^2
+{\ep}\int_0^t \|\na \psi\|^2
\le C \cdot( \| \psi_0\|^2 + \|{\varphi_0}\|_w^2 )+ C\cdot{M(t)}\cdot \int_0^t 
 \|{\na \psi}\|_w^2.
\end{equation}


\subsection{First order estimate}\label{sec_first} \ \\

From now on, we estimate the derivatives of $\vp$ and $\psi$.

\begin{lemma}\label{lemma1_0} If the positive constants $\ep_0,  \delta_0 $ are sufficiently small, then there exists a constant 
  $C>0$ such that 
 for any $t\in[0,T]$, 
\begin{equation}\begin{split}\label{ineq:lem1_}
&\| \na \vp\|_w^2  + \| \na \ps\|^2+ \int_0^t \| \na^2 \vp\|_w^2 
+ \ep\int_0^t \|\na^2 \psi\|^2\\
&\le C( \| \na \vp_0\|_w^2 + \| \na \ps_0\|^2+  \| \psi_0\|^2 + \| \varphi_0\|_w^2 ) 
+  C \int_0^t \int {N} |\na \ps|^2 + C {M(t)} \int_0^t \int \fn{\abs{\na \ps}}.
\end{split}\end{equation}\end{lemma}
 
\begin{proof}
We differentiate \eqref{main_eq} in $z$ to get
\begin{align*}
\vp_{tz} - s\vp_{zz} - \Del \vp_z & = N' \na\psi + N \na \psi_z + P' \dvg\vp + P \dvg \vp_z + \na \psi_z \dvg \vp + \na\psi \dvg \vp_z,\\
\psi_{tz} -s\psi_{zz} -\ep \Del \ps_z & = -2\ep  (P \cdot\na \ps)_z -\ep  (|\na \ps|^2)_z  + \dvg \vp_z.
\end{align*}
We multiply  $\fn{\varphi_z}$ to $\vp$ equation and $\psi_z$ to the $\psi$ equation from above and do integration by parts to get
\begin{align*}
&\frac 12 \ddt \lr{ \int \fn {\abs{\vp_z} } + \abs{\ps_z}} + \int \sum_{ij}\fn{ \abs{\pa_j \vp^i_{z}}}
+ \ep\int |\na \psi_z|^2\\
 & = 
 { \frac 12 \int \abs{\vp_z} \left( \fn{1} \right)''- \frac s2 \int \abs{\vp_z} \Fn' }
 \\
 & +\int \fn{N'} \na \ps \vp_z  +\int \fn{P'} \dv \vp \vp_z + \int \fn{P} \dv \vp_z \vp_z\\
 &+  \int \na \ps_z \dv \vp \fn{\vp_z} + \int \na\ps \dv \vp_z \fn{\vp_z}\\
& -\underbrace{ 2\ep \int (P\cdot \na\ps)_z \psi_z }_
{=  2\ep \int (\calP \ps_z)_z \psi_z}
- \ep \int (|\na \ps|^2)_z \ps_z.
\end{align*}
Similarly,  
we get
\begin{align*}
&\frac 12 \ddt \lr{ \int \fn {\abs{\vp_y} } + \abs{\ps_y}} +\int \sum_{ij}\fn{ \abs{\pa_j \vp^i_{y}}} 
+ \ep\int |\na \psi_y|^2\\
   &=  
   { \frac 12 \int \abs{\vp_y} \left( \fn{1} \right)'' - \frac s2 \int \abs{\vp_y} \Fn' }
 \\
   &+\int \fn{P} \dv \vp_y \vp_y \\
   &+ \int \na \ps_y \dv \vp \fn{\vp_y} + \int \na\ps \dv \vp_y \fn{\vp_y}\\
  & -\underbrace{ 2\ep \int (P\cdot \na\ps)_y \psi_y }_
{=  2\ep \int \calP \ps_{zy} \psi_y}
- \ep \int (|\na \ps|^2)_y \ps_y.
\end{align*}

First, we observe
\begin{align*}
\frac 12 \int |\na\vp|^2 \lr{ \fn{1}}''\underbrace{-\frac s 2 \int |\na\varphi|^2 \left( \fn{1}\right)'}_{\leq0}
&\le -\int \na\vp\cdot\na\vp_z  \left( \fn{1}\right)' 
\le C \int |\na\vp||\na\vp_z |  \left( \fn{1}\right) \\&
\le  \frac{1}{8}\int  |\na^2\vp |^2\left( \fn{1}\right)+C \int |\na\varphi|^2  \left( \fn{1}\right).
 \end{align*}

We estimate the quadratic terms as follows:
\begin{align*}
\int \fn{P} \dv \vp_z \vp_z+\int \fn{P} \dv \vp_y \vp_y  \le
 \| P\|_{L^{\infty}} (\|  \frac{\na\vp_z}{ \sqrt{N}} \|\| \frac{\vp_z}{\sqrt{N}}\| 
+\|    \frac{\na\vp_y}{ \sqrt{N}} \|\| \frac{\vp_y}{\sqrt{N}}\| )
 \le
\frac{1}{4}\|  \frac{\na^2\vp}{ \sqrt{N}} \|^2+C\| \frac{\na \vp}{\sqrt{N}}\|^2,
\end{align*}
\begin{align*}
\int \fn{N'} \na \ps \vp_z & \le \|s+\PP\|_{L^{\infty}}\| {\sqrt N}{\na\ps}\| \| \frac{\vp_z}{ \sqrt{N}} \|
\le C\| {\sqrt N}{\na\ps}\|^2+ C \| \frac{\na\vp}{ \sqrt{N}} \|^2\quad \mbox{and}\\
\int \fn{P'} \dv \vp \vp_z  &\le 
C\| P'\|_{L^{\infty}} \| \frac{\na \vp}{\sqrt N} \|^2
\le 
C\|\frac{\na \vp}{\sqrt N}  \|^2.
\end{align*}

The sum of all cubic terms is bounded by 
\begin{align*}
C\int  | \na\ps|| \na\vp| \fn{|\na^2\vp|} 
+C\int  | \na\ps|| \na\vp|^2 \underbrace{|({1}/{N})'|}_{\leq C/{N}}
& \le C \sqrt{M(t)} ( \| \fnn{\na^2 \vp} \|^2  + \| \fnn{\na \vp} \|^2)\\
& \le \frac{1}{4}  \| \fnn{\na^2 \vp}\|^2 + C\| \fnn{\na \vp}\|^2
\end{align*}
by  $ \| \na\ps\|_{L^{\infty}} \le C\sqrt{M(t)}\le C\sqrt{\delta_0}$ and by assuming $\delta_0>0$ small enough.\\

For the $\ep$-terms, 
 we estimate
\begin{align*}
&-2\ep \int (\calP \ps_z)_z \psi_z  - 2\ep \int \calP \ps_{zy} \psi_y=
  -2\ep \Big(\int \calP' \ps_z \psi_z
+\int \calP \ps_{zz} \psi_z 
 +\int \calP \ps_{zy} \psi_y \Big)\\
  &\leq C\ep \|P\|_{L^\infty}\|\na\ps\|\|\na^2\ps\|+C\ep \|P'\|_{L^\infty}\|\na\ps\|^2  \\
    &\leq C\ep  \|\na\ps\|^2+\frac{\ep}{4}\|\na^2\ps\|^2   
\end{align*}
and

\begin{align*}
&- \ep \int (|\na \ps|^2)_z \ps_z  
- \ep \int (|\na \ps|^2)_y \ps_y=
  \ep \int (|\na \ps|^2) \ps_{zz}  
+ \ep \int (|\na \ps|^2) \ps_{yy}\\
&\leq  C\ep \sqrt{M(t)} \int |\na \ps||\na^2 \ps|
\leq  C\ep \|\na \ps\|^2+\frac{\ep}{4}\|\na^2 \ps\|^2.
\end{align*}
Adding up all the estimates above, we have 
\begin{align*}
&\frac 12 \ddt \lr { \int \fn{\abs{\na \vp}} + \int \abs{\na \ps} } + \frac 14\int   \fn{\abs{ \na^2\vp } }
+ \frac{\ep}{4}\int |\na^2 \psi|^2
\le  
C\|{\sqrt{N}} {\na \ps}\|^2  + C   \| \frac{\na \vp}{\sqrt{N}}\|^2 
 +C\ep \|\na \ps\|^2.
\end{align*}
After integration in time, thanks to  \eqref{eq_lemma0_}, we can control the last two terms above so that we arrive at \eqref{ineq:lem1_}.
\end{proof}

\begin{lemma}\label{lemma1_1}  If the positive constants $\ep_0,  \delta_0 $ are sufficiently small, then there exists a constant 
  $C>0$ such that 
 for any $t\in[0,T]$, 
 \begin{align} \label{claim1_lemma1_}
& \int_0^t \int N |\na \ps|^2 +\ep\int_0^t \int  |\na^2 \ps|^2 \le C \lr{ \| \na \ps_0 \|^2 + \| \ps _0\|^2 + \| \vp_0\|_w^2} 
+ C\sqrt{M(t)} \int_0^t  \int \fn{ |\na \ps |^2}.
\end{align}
\end{lemma}
 
\begin{proof}

Multiplying $\na\ps$ to the $\vp$-equation, 
we have 
\begin{equation}\begin{split}\label{eq_lemma0__first_claim_intermediate}
 N |\na \ps|^2 &=  {\vp_t \cdot \na \ps}  - s\vp_z \cdot\na \ps - {\Del \vp \cdot\na \ps}  
-(\dv \vp)P\cdot\na\ps
- \dv \vp | \na \ps|^2\\
&=  
(\vp \cdot\na \ps)_t - \vp \cdot\na \ps_t 
 - s\vp_z \cdot\na \ps -\underbrace{\Del \vp \cdot\na \ps}_{(*)} 
-(\dv \vp)P\cdot\na\ps
- \dv \vp | \na \ps|^2.
\end{split}\end{equation}

For the second term $\vp \cdot\na \ps_t $, we use the $\psi$ equation (after taking $\na$):
$$  
\na \ps_t=s \na \ps_z + \na (\dv \vp)
  +\ep \Del \na\ps  -2\ep \na( P \cdot\na \ps) -\ep  \na(|\na \ps|^2) 
$$

 For $(*)$, we observe that
\[(*) 
 = (\dv \vp)_z \ps_z + ( \vp^1_{yy} - \vp^2_{zy}) \ps_z + ( \dv \vp)_y \ps_y 
+ (\vp^2_{zz} - \vp^1_{zy}) \ps_y.\]
Integrating by parts for $\int(*)$ gives
\[ \int(*)
 =  \int  (\dv \vp)_z \ps_z + ( \dv \vp)_y \ps_y
  =  \int  \na(\dv \vp) \cdot \na\ps.\]
 By using the $\ps$-equation which have $\na\cdot\vp$ on its right-hand side,
we get
\begin{align*} \int \na (\dv \vp)\cdot\na \ps &= \int \na (\ps_{t} - s \ps_{z}
 -\ep \Del \ps +2\ep  P \cdot\na \ps+\ep  |\na \ps|^2
)\cdot \na  \ps \\&= \frac 12 \ddt \int |\na\ps|^2
+\int \na( 
 -\ep \Del \ps +2\ep  P \cdot\na \ps+\ep  |\na \ps|^2
)\cdot\na \ps.
\end{align*}

Thus integration on \eqref{eq_lemma0__first_claim_intermediate} gives us
\begin{align*}
\int N |\na \ps|^2 =& 
\ddt \int \vp \cdot\na \ps - \int s \vp \cdot \na \ps_z -\int\vp\cdot (\na (\dv \vp)
+ \ep \Del \na\ps  -2\ep \na( P \cdot\na \ps) -\ep  \na(|\na \ps|^2) )\\ \nonumber
&- \int s \vp_z\cdot \na \ps \\ \nonumber
&- \frac 12 \ddt\int
|\na \psi|^2
+\int \na( 
 \ep \Del \ps -2\ep  P \cdot\na \ps-\ep  |\na \ps|^2
)\cdot\na\ps
\\ \nonumber
&- \int (\dv \vp)P\cdot\na\ps   - \int  \dv \vp | \na \ps|^2. 
  \end{align*}
 We rearrange the above to get
 \begin{align}\label{ineq_claim1_lemma1}
\int N |\na \ps|^2  
=& \ddt \int \vp \cdot\na \ps - \frac 12 \ddt \int  |\na\ps|^2 \\& \nonumber + \int | \dv \vp|^2 - \int (\dv \vp)P\cdot\na\ps - \int  \dv \vp | \na \ps|^2\\ \nonumber
&+\ep\int \na( 
  \Del \ps -2  P \cdot\na \ps-   |\na \ps|^2
)\cdot\na\ps - \ep\int\vp\cdot(   \Del \na\ps  -2 \na( P \cdot\na \ps) -  \na(|\na \ps|^2) )\\& \nonumber=(I)+(II)+(III).
 \end{align}
 For the first term $(I)$, after integrating in time, we get
 \begin{align*}
 \int_0^t(I)=\int_0^t\lr{\ddt \int \vp \cdot\na \ps - \frac 12 \ddt \int  |\na\ps|^2}\leq C(\|\vp(t)\|^2 + \|\na\ps_0\|^2+\|\vp_0\|^2)&
 \end{align*} by 
$ \int  \vp \na\ps \le C\| \vp \| ^2 + \frac 12   \| \na \ps\|^2$.

For the second term $(II)$, we estimate 
 \begin{align*}
  \int | \dv \vp|^2-\int (\dv \vp)P\cdot\na\ps 
 & \le  C\int \fn{ | \na \vp|^2} + \frac 14 \int  N |\na \ps|^2, \\
  \int  |\dv \vp || \na \ps|^2 & \le C\sqrt{M(t)}    \int {| \na \ps|^2} \le C\sqrt{M(t)}    \int \fn{| \na \ps|^2}
 \end{align*}
 by bounding $\| \dv\vp \|_{L^{\infty}}  \le C \| \dv\vp \|_{H^2}\le C \sqrt{M(t)}$. 

For the $\ep$-term $(III)$,  
 by bounding $\| \na\psi \|_{L^{\infty}}  \le C \| \na\psi \|_{H^2}\le C \sqrt{M(t)}\le C \sqrt{\delta_0}\leq C$, 
 we estimate \begin{align*}
 \ep \int \na(    \Del \ps -2   P \cdot\na \ps-   |\na \ps|^2)\cdot\na\ps &=
-\ep\int |\na^2\ps|^2
-\ep \int \na(     2   P \cdot\na \ps+  |\na \ps|^2)\cdot\na\ps\\
&\leq-\ep\|\na^2\ps\|^2 +C\ep  
\int(|\na\psi|^2+|\na^2\psi||\na\psi|+\underbrace{|\na\psi|^2}_{\leq C|\na\psi|}|\na^2\psi|)
\\
&\leq-\ep\|\na^2\ps\|^2 +C\ep  \|\na\ps\|^2+\frac{\ep}{4}\|\na^2\ps\|^2
+C\ep  \|\na\ps\|^2\\
&\leq-\frac{3}{4}\ep\|\na^2\ps\|^2 +C\ep  \|\na\ps\|^2 
\end{align*} and 
\begin{align*}
 - \ep\int\vp\cdot(   \Del \na\ps  -2 \na( P \cdot\na \ps) -  \na(|\na \ps|^2) ) 
&\leq  C\ep\int(|\na\vp||\na^2\ps|  + |\na\vp||\na \ps| +|\na\vp|\underbrace{|\na \ps|^2}_{\leq C|\na\psi|})\\
&\leq \frac\ep 4\|\na^2\ps\|^2 +C\ep  \|\na\vp\|^2+C\ep  \|\na\ps\|^2\\
&\leq \frac\ep 4\|\na^2\ps\|^2 +C\ep  \|\na\vp\|_w^2+C\ep  \|\na\ps\|^2.
\end{align*}

 \begin{remark}
The key idea is to observe that the chemical $c$ is consumed by the cells $n$ in the  system  \eqref{KS}. More precisely, the negative sign of the term $cn$ in the right-hand side of the $c$-equation in \eqref{eq:main}, which is related to 
 the positive sign of the term $\na n$ in the right-hand side of the $p$-equation in \eqref{nq},
 is passed down to the signs of terms $-\frac 12 \int_0^t \frac{d}{dt}\| \na \psi\|^2 $ in ($I$) and $-\ep \| \na^2\psi \|^2 $ in ($III$). 
\end{remark}

Integrating \eqref{ineq_claim1_lemma1} in time, we get
 \begin{align*} 
\int_0^t\int& N |\na \ps|^2  
\leq 
C(\|\vp(t)\|^2 + \|\na\ps_0\|^2+\|\vp_0\|^2)\\
& +\int_0^t\Big(C  \|\na\vp\|_w^2+ \frac 14 \int  N |\na \ps|^2  +C\sqrt{M(t)}    \int \fn{| \na \ps|^2} - \frac\ep 2\|\na^2\ps\|^2 +C\ep  \|\na\ps\|^2\Big).
 \end{align*}


By the estimate \eqref{eq_lemma0_}, we have 
\eqref{claim1_lemma1_}.  
\end{proof}

\begin{lemma}\label{lemma1_2}  If the positive constants $\ep_0,  \delta_0 $ are sufficiently small, then there exists a constant 
  $C>0$ such that 
 for any $t\in[0,T]$, 
 \begin{align} 
& \int \fn {|\na \ps|^2} +   \int _0^ t \int \fn {|\na \ps|^2 }
+ \ep \int_0^t\int    \fn{|\na^2\ps|^2}
\le  C(  \| \na \ps_0 \|_w^2 + \| \ps _0\|^2 + \| \vp_0\|_{w}^2) + C\int_0^t \int \fn{ |\na^2\vp|^2}. \label{claim2_lemma1_}
\end{align}
\end{lemma}
 
\begin{proof}
First we take $\na$ to the $\ps$-equation then multiply by  $ w \na \ps$ to get
\begin{align*} 
&\frac 12 (w |\na \ps|^2)_t - \frac s2 ( w |\na \ps|^2)_z + \frac s2 w' |\na \ps|^2 \\
&= w \na (\dv \vp) \cdot \na \ps 
+\underbrace{\ep w \na \ps\cdot\lr{\Del \na\ps -2  \na( P \cdot\na \ps) -   \na(|\na \ps|^2)}}_{\ep\mbox{-terms}}.
\end{align*}
 
We  assume that $\ep_0>0$ is  smaller than $\ep_1>0$ in Lemma \ref{lem_center}.
 Then,  by  \eqref{rem_center},  there exists a point $z_0\in\bbR$  such that
  \ben\begin{split}
&\frac{w'(z)}{w(z)}\geq \frac{s}{2}\quad \mbox{for } z\geq z_0\quad \mbox{and}\quad w(z)\leq \frac{4}{s^2}\leq \frac{16}{s^4} N \quad \mbox{for } z\leq z_0.
 \end{split}\een

Integrating on each half strip (notation : $\int_{z>z_0}f:=\int_{z_0}^\infty\int_0^\la f (z,y,t)dy dz$)  and in time, we get 
\begin{align*}
\frac 1 2\int _{z>{z_0}} w |\na \ps|^2 &
 \le \frac 1 2 \int_{z>{z_0}} w |\na \ps_0|^2 + \int_0^t \int_{z>{z_0}} w \na (\dv \vp) \cdot \na \ps -
\frac{s}{2}\int_0^t \int _{z>{z_0}} \underbrace{w'}_{\geq \frac s 2 w} |\na \ps|^2 \\&\quad -\frac s2\int_0^t  \int_0^\la w |\na \ps|^2({z_0}, y) dy+\int_0^t \int_{z>{z_0}}{\ep\mbox{-terms}}\\
&
 \le \frac 1 2 \int_{z>{z_0}} w |\na \ps_0|^2 + \int_0^t \int_{z>{z_0}} w |\na (\dv \vp)|  |\na \ps| -
\frac{s^2}{4}\int_0^t \int _{z>{z_0}} w |\na \ps|^2 \\&\quad -\frac s2\int_0^t  \int_0^\la w |\na \ps|^2({z_0}, y) dy+\int_0^t \int_{z>{z_0}}{\ep\mbox{-terms}}\\
& \le  \frac 1 2 \int _{z>{z_0}} w |\na \ps_0|^2 -\frac{s^2}{8} \int_0^t \int_{z>{z_0}} w |\na \ps|^2 +
  C \int_0^t \int_{z>{z_0}} w |\na(\na \cdot \vp)|^2 \\&\quad -\frac s2\int_0^t  \int_0^\la w |\na \ps|^2({z_0}, y) dy
  +\int_0^t \int_{z>{z_0}}{\ep\mbox{-terms}}
  \end{align*} and 
  \begin{align*}
\frac 1 2 \int _{z<{z_0}} w |\na \ps|^2 \le & \frac 1 2\int_{z<{z_0}} w |\na \ps_0|^2 + \int_0^t \int_{z<{z_0}} \underbrace{w}_{\leq C}| \na (\dv \vp)| | \na \ps|
  \underbrace{-\frac{s}{2}\int_0^t \int _{z<{z_0}} w' |\na \ps|^2}_{\leq 0}
\\& +\frac s2\int_0^t  \int_0^\la w |\na \ps|^2({z_0}, y) dy +\int_0^t \int_{z<{z_0}}{\ep\mbox{-terms}}\\
\le& \frac 1 2\int_{z<{z_0}} w |\na \ps_0|^2 + C\int_0^t \int_{z<{z_0}} |\na (\dv \vp)|| \na \ps |
 \\&  +\frac s2\int_0^t  \int_0^\la w |\na \ps|^2({z_0}, y) dy +\int_0^t \int_{z<{z_0}}{\ep\mbox{-terms}}\\
  \le&  \frac 1 2\int_{z<{z_0}} w |\na \ps_0|^2 + C\int_0^t \int_{z<{z_0}} N |\na \ps|^2  + C\int_0^t \int_{z<{z_0}} \fn{ |\na^2 \vp|^2}
 \\&  +\frac s2\int_0^t  \int_0^\la w |\na \ps|^2({z_0}, y) dy +\int_0^t \int_{z<{z_0}}{\ep\mbox{-terms}} .
\end{align*}
Adding   the above two estimates, we get
\begin{align*}
&\frac 1 2 \int w |\na \ps|^2 + \frac{s^2}{8} \int_0^t \int_{z>z_0} w |\na \ps|^2\\
& \le \frac 1 2 \int w |\na \ps_0|^2 +  C\int_0^t \int N|\na\ps|^2 + 
C\int_0^t \int w |\na^2 \vp|^2  +\int_0^t \int {\ep\mbox{-terms}}.
\end{align*}
Adding $ \frac{s^2}{8} \int_0^t \int_{z<z_0} w |\na \ps|^2$ to the both sides  and noting $w\leq C N$ on $\{z<z_0\}$, we get
\begin{align*}
&\frac 1 2 \int w |\na \ps|^2 + \frac{s^2}{8} \int_0^t \int w |\na \ps|^2\\
& \le \frac 1 2 \int w |\na \ps_0|^2 +  C\int_0^t \int N|\na\ps|^2 + 
C\int_0^t \int w |\na^2 \vp|^2  +\int_0^t \int {\ep\mbox{-terms}}\\
&\le C ( \| \na \ps_0\|_w^2 + \| \ps_0\|^2 + \| \vp_0\|_w^2) + C\underbrace{\sqrt{{M(t)}}}_{\leq\delta_0} \int_0^t \int \fn{ |\na \ps|^2}
+ C \int_0^t \int \fn{ |\na^2 \vp|^2}+\int_0^t \int {\ep\mbox{-terms}}  
\end{align*}
where  we used the previous estimate \eqref{claim1_lemma1_} for the last inequality.   

For the $\ep ${-terms},  we estimate
\begin{align*}
&\int {\ep \mbox{-terms}}=\ep\int { w \na \ps\cdot\lr{\Del \na\ps -2  \na( P \cdot\na \ps) -   \na(|\na \ps|^2)}}\\
&=\ep\int \Big({ -w|\na^2\ps|^2 -w' \na \ps\cdot  \na\ps_z
- w \na \ps\cdot\lr{  2  \na( P \cdot\na \ps)} 
+  
(w'\psi_z+w\Del\ps)
|\na \ps|^2}\Big)\\
&\leq-\ep\int {  w|\na^2\ps|^2 +C\ep\int \Big( w |\na \ps|| \na^2\ps| + w| \na \ps|( |\na^2\ps|+|\na\ps| )
+w |\na \ps|    |\na \ps|^2}+w|\na^2\ps||\na\ps|^2\Big)\\
&\leq-\frac \ep 4\int   w|\na^2\ps|^2 +C\ep\int w |\na \ps|^2 
 +C\ep\sqrt{M(t)}\int w |\na \ps|^2
\leq-\frac \ep 4\int   w|\na^2\ps|^2 +C{\ep}\int w |\na \ps|^2 
\end{align*} where we used the estimate 
$|\frac{w'}{w}|=|s+\PP|\leq s$.

In sum, we have 
\begin{align*}
&\frac{1}{2}\int w |\na \ps|^2 +  (\frac{s^2}{8}-C(\ep_0+\sqrt{\delta_0})) \int_0^t \int w |\na \ps|^2 +\frac{ \ep}{4} \int_0^t\int   w|\na^2\ps|^2 \\
&\le C ( \| \na \ps_0\|_w^2 + \| \ps_0\|^2 + \| \vp_0\|_w^2) 
+ C \int_0^t \int \fn{ |\na^2 \vp|^2}.  
\end{align*} Then, by making $\ep_0>0$ and $\delta_0>0$ small enough, it proves the estimate \eqref{claim2_lemma1_}.

\end{proof}
Up to now, we have proved the following first order  energy estimate, which is closed except  that we assumed that higher order norms are small by $M(T)\leq\delta_0$:
\begin{lemma}\label{lemma01_} If the positive constants $\ep_0,  \delta_0 $ are sufficiently small, then there exists a constant 
  $C>0$ such that 
 for any $t\in[0,T]$,
  \begin{equation}\begin{split}\label{eq_lemma01_}
\| \vp\|_{1,w}^2 + \| \ps\|^2    + \| \na \ps\|_w^2+  \int_0^t  \sum_{l = 1,2} \| \na^{l} \vp\|_w^2
+\int_0^t   \| \na \ps\|_w^2
+\ep\int_0^t   \| \na^{2} \ps\|_w^2
\\
\quad  \le
 C ( \|  \vp_0\|_{1,w}^2  +
\| \ps_0\|^2 + \| \na \ps_0\|_w^2 ). 
\end{split}\end{equation}
\end{lemma}

\begin{proof}
Plugging the estimates \eqref{claim1_lemma1_} and \eqref{claim2_lemma1_} into \eqref{ineq:lem1_}, 
we  have
\begin{align*}
 &\| \na\vp\|_w^2 + \| \na\ps\|^2 + \int_0^t \| \na^2 \vp\|_w^2  + \ep \int_0^t\|\na^2\ps\|^2 \\
 &\le  C ( \| \na \ps_0\|_w^2 + \| \ps_0\|^2 + \| \vp_0\|_{1,w}^2) + C\sqrt{{M(t)}}   \int_0^t \int \fn{|\na^2 \vp|^2}
\end{align*} which gives us 
\begin{align}\label{ineq:lem1_closed}
 &\| \na\vp\|_w^2 + \| \na\ps\|^2 + \int_0^t \| \na^2 \vp\|_w^2  + \ep \int_0^t\|\na^2\ps\|^2
\le  C ( \| \na \ps_0\|_w^2 + \| \ps_0\|^2 + \| \vp_0\|_{1,w}^2) 
\end{align}
if we assume $\delta_0>0$ small enough.
In addition, from the estimate \eqref{claim2_lemma1_} together with the above estimate \eqref{ineq:lem1_closed}, we get
\begin{align}\label{claims_lemma1_}
\int \fn{ |\na \ps|^2} + \int_0^t \int \fn{ |\na \ps|^2} + 
 \ep\int_0^t \int \fn{ |\na^2 \ps|^2} \le C ( \|  \vp_0\|_{1,w}^2   + \| \na \ps_0\|_w^2
+ \| \ps_0\|^2).
\end{align}
 By adding the estimate \eqref{eq_lemma0_} to the above estimates \eqref{ineq:lem1_closed} and \eqref{claims_lemma1_} and by
assuming $\delta_0$ small enough,  
we have \eqref{eq_lemma01_}.
 
\end{proof}

\subsection{Higher order estimate}\label{sec_higher} \ \\

 To finish the proof of Proposition \ref{uniform_}, we need to do   similar  energy estimates up to the third  order derivatives. 
 We collect all the higher order estimates  into  the following SINGLE lemma, which can be proved in a similar way as we did for the first order estimate in
   Lemma \ref{lemma1_0}, \ref{lemma1_1}, \ref{lemma1_2} and \ref{lemma01_} in
the last subsection. Here we present its proof in detail for readers' convenience.

\begin{lemma}\label{lemma23_} If the positive constants $\ep_0,  \delta_0 $ are sufficiently small, then there exists a constant 
  $C>0$ such that 
 for any $t\in[0,T]$ and  
for $k=2, 3$, we have 
\begin{align*}
&\| \na^k \vp\|^2_w + \| \na^k \ps\|^2_w +   \int_0^t  { \|\na^{k+1}\vp\|_w^2}
+ \int_0^t  { \|\na^{k}\ps\|_w^2}
+  \ep\int_0^t  { \|\na^{k+1}\ps\|_w^2}\\
&\le  C (  \| \vp_0\|_{k,w}^2+
 \| \na \ps_0\|_{k-1,w}^2+ \| \ps_0\|^2).
\end{align*}
\end{lemma}

\begin{proof}[Proof of Lemma \ref{lemma23_}]
Differentiating the $\vp,\ps$ equations $i+j$ times  in $y$ or $z$, we have 
\begin{align*}
&\ipa \jpa \vp_t - s \ipa \jpa \vp_z - \Del \ipa \jpa \vp \\
&\phantom{\ipa \jpa \ps_t - s \ipa\jpa \ps_z} =
(    \ipa\jpa (N \na \ps) - N \ipa\jpa \na\ps  )  
  + N \na  \ipa\jpa \ps  + \ipa\jpa( P \dv \vp) + \ipa \jpa( \dv \vp \na \ps)\\
&\ipa \jpa \ps_t - s \ipa\jpa \ps_z -\ep \Del \ipa\jpa\ps  = -2\ep  \ipa\jpa(P \cdot\na \ps) -\ep  \ipa\jpa(|\na \ps|^2)+  \dv \ipa\jpa \vp.
\end{align*}
Thus we get
\begin{align} \label{higher_}
& \frac 12 \ddt \int  \lr{ \fn{ | \ipa \jpa \vp|^2} + | \ipa \jpa \ps|^2 } 
+  \int \fn { | \na \ipa\jpa \vp |^2}
+  \ep\int  { | \na \ipa\jpa \ps|^2}\\
 &=
  {  \frac 12 \int \abs{\ipa \jpa \vp} \left( \fn{1} \right)''- \frac s2 \int \abs{\ipa \jpa \vp} \Fn'}
 \nonumber\\
 &\underbrace{+\int  (    \ipa\jpa (N \na \ps) - N \ipa\jpa \na\ps  )\cdot  \fn { \ipa \jpa \vp} + \ipa \jpa (P \dv \vp) \cdot\fn{ \ipa\jpa \vp} }_{\mbox{ Quadratic term }}\nonumber\\
& \underbrace{+ \ipa\jpa ( \dv \vp \na \ps) \fn{ \ipa \jpa \vp}}_{\mbox{ Cubic term }} \nonumber\\
& \underbrace{ -
2\ep \int \ipa\jpa(\calP \ps_z) \ipa\jpa\psi
- \ep \int \ipa\jpa(|\na \ps|^2) \ipa\jpa\ps}_{\ep-\mbox{term}}. \nonumber
\end{align}

$\bullet$ Case $ k= i+j =2$\\

First, we estimate
\begin{align*}
&\frac 12 \int |\na^2\vp|^2 \lr{ \fn{1}}''-\frac s 2 \int |\na^2\varphi|^2 \left( \fn{1}\right)'
\le -\int \na^2\vp\cdot\na^2\vp_z  \left( \fn{1}\right)'
\\
&
\le  C\int |\na^2\vp||\na^2\vp_z |  \left( \fn{1}\right) 
\le  \frac{1}{4}\int  |\na^3\vp |^2\left( \fn{1}\right)+C \int |\na^2\varphi|^2  \left( \fn{1}\right).
 \end{align*}


What it follows, we do not distinguish $\pa_y$ and $ \pa_z$ derivatives.\\
The quadratic terms are symbolically
\begin{align*}
 \fn{N''} \na \ps \na^2 \vp, \quad  \fn{N'} \na^2 \ps \na^2 \vp,\quad \\
 \fn{P''} \na \vp \na^2 \vp, \quad   \fn{P'} \na^2 \vp \na^2 \vp  \quad \mbox{ and } \quad \fn{P} \na^3 \vp \na^2 \vp.\\
\end{align*}
We recall $P=(\PP,0)$, Lemma \ref{lem_NP}, and \eqref{NP_relation}: 
\[ |{N^{ (k)}}| < C, \quad   |{P^{(k)}}| <C,\quad \mbox{ for } 0\leq k\le 2,\quad 
 \Big|\fn{N'}\Big|=|\PP+s|\leq C, \mbox{ and }\Big|(\fn{N'})'\Big|=|\PP'|\leq C.\]
So the quadratic   terms are estimated by 
\begin{align*}
& C\int \Big[\Big|\fn{N''}\Big|| \na \ps|| \na^2 \vp| +
 \Big|\fn{P''} \Big||\na \vp ||\na^2 \vp|+ \Big|\fn{P'} \Big||\na^2 \vp ||\na^2 \vp|\Big]\\
 &\leq  C \|  \frac{\na^2\vp }{ \sqrt{N}} \|\lr{\| \frac{\na\ps }{\sqrt{N}}\| 
 +\| \frac{\na\vp }{\sqrt{N}}\|+\| \frac{\na^2\vp }{\sqrt{N}}\|}\\
 &\leq  C \lr{\| \frac{\na\ps }{\sqrt{N}}\|^2
 +\| \frac{\na\vp }{\sqrt{N}}\|^2+\| \frac{\na^2\vp }{\sqrt{N}}\|^2},
\end{align*} 
\begin{align*}
C\int \Big|\frac{P}{N}\Big|  |\na^3 \vp| |\na^2\vp|   \le
 C \|  \frac{\na^3\vp }{ \sqrt{N}} \|\| \frac{\na^2\vp }{\sqrt{N}}\|  \le
 \frac{1}{8} \|  \frac{\na^3\vp }{ \sqrt{N}} \|^2 +C\| \frac{\na^2\vp }{\sqrt{N}}\|^2
\end{align*} and 
\begin{align*}
& C\Big|\int \fn{N'} \na^2 \ps \na^2 \vp\Big|\leq C \int  | \na \ps| |\na^2 \vp|
+C\int   |\na \ps ||\na^3 \vp|\\
 &\leq  C \|  {\na\ps } \|\lr{\|  {\na^2\vp } \|+\|  {\na^3\vp } \|}\leq 
 C \|  {\na\ps } \|^2+ C\|  {\na^2\vp } \|^2+\frac 1 8 \|  {\na^3\vp } \|^2
\end{align*} where we used integration by parts for the last estimate.

The cubic terms 
 are symbolically written as 
$ \na^2 (\na \vp \na \ps) \fn{\na^2\vp}$. By using integration by parts once, it can be written as 
\begin{align*}
& \na  (\na \vp \na \ps) \fn{\na^3\vp} \quad \mbox{and } \quad \na  (\na \vp \na \ps) {\na^2\vp}
{(\frac{1}{N})'} .
\end{align*} So by assuming   $M(t)$ small enough, these terms are estimated by
\begin{align*}
&   C\int (|\na^2 \vp|| \na \ps|+|\na \vp ||\na^2 \ps|) \fn{|\na^3\vp|} \leq
C\sqrt{M(t)}\int (|\na^2 \vp| + |\na^2 \ps|) \fn{|\na^3\vp|}\\
&\leq  C   \|\frac{\na^2 \vp}{\sqrt{N}}\|^2+ C \sqrt{M(t)}\|\frac{\na^2 \ps}{\sqrt{N}}\|^2  +\frac{1}{8}\|\frac{\na^3\vp}{\sqrt{N}}\|^2
\end{align*} and

\begin{align*}
&   C\int(|\na^2 \vp ||\na \ps|
+|\na \vp|| \na^2 \ps|){\fn{|\na^2\vp|}}\leq C  \|\frac{\na^2 \vp}{\sqrt{N}}\|^2+ C \sqrt{M(t)}\|\frac{\na^2 \ps}{\sqrt{N}}\|^2.
\end{align*} Note that we  used $M(t)$ for   $ \| \na\vp\|^2_{L^{\infty}} \le M(t)$ and  $ \|\na \ps\|^2_{L^{\infty}} <M(t)$.

 For the $\ep$-terms, 
 we can write them  symbolically:
$$ 
    \ep \int \na^2(\calP \ps_z) \na^2\psi  \quad \mbox{and } \quad
 \ep \int\nabla^2(|\na \ps|^2) \nabla^2\ps.
$$
After integration by parts, we can estimate them by
\begin{align*}
 & C\ep \int |\na(\calP \ps_z)| |\na^3\psi|\leq
 C\ep \int (|\na \ps| +|\na^2 \ps|) |\na^3\psi|
 \leq  C\ep ( \|\nabla\ps\|^2+ \|\nabla^2\ps\|^2) +\frac{\ep}{4}\|\nabla^3\ps\|^2 
\end{align*}
 and
\begin{align*} C\ep \int|\nabla(|\na \ps|^2)| |\nabla^3\ps|&\leq C\ep\int|\nabla\ps||\nabla^2\ps| |\nabla^3\ps|\\
&\leq C\ep  \sqrt{M(t)}\int |\nabla^2\ps| |\nabla^3\ps|\leq  C\ep 
 \|\nabla^2\ps\|^2 +\frac{\ep}{4}\|\nabla^3\ps\|^2.
\end{align*}
 
 Up to now, by Lemma \ref{lemma01_},   \eqref{higher_} becomes  \begin{align}\label{beforeclaims_lemma23_} &\int \fn{ |\na^2 \vp|^2} +\int |\na^2 \ps|^2 + \int_0^t \int \fn{ |\na^3 \vp|^2} +
 \ep\int_0^t \int{ |\na^3 \ps|^2}\\
 &\le C( \|  \vp_0\|_{2,w}^2  +
 \| \na \ps_0\|_{1,w}^2+ \| \ps_0\|^2) 
 +C \sqrt{M(t)}\int_0^t \int  \fn{|\na^2 \ps|^2}.  \nonumber 
 \end{align} This estimate is the second order version of  Lemma \ref{lemma1_0}.\\
 
 Now we claim the following two estimates which are the second order versions of Lemmas \ref{lemma1_1} and
 \ref{lemma1_2}:
 
 \begin{align} 
 \label{claim1_lemma23_}
& \int_0^t \int N |\na^2 \ps|^2 + \ep\int _0^ t \int   {|\na \Del \ps|^2 } \\&\qquad\qquad\qquad \qquad \le C \lr{   \|  \vp_0\|_{1,w}^2  +
 \| \na \ps_0\|_w^2+ \| \ps_0\|^2+\|\Del\ps_0\|^2} + C\sqrt{M(t)} \int_0^t  \int \fn{ |\na^2 \ps |^2}\nonumber \quad\mbox{ and} \\
& \int \fn {|\na^2  \ps|^2} +   \int _0^ t \int \fn {|\na^2 \ps|^2 } 
+ \ep\int _0^ t \int \fn {|\na^3 \ps|^2 } \le  C    ( \| \na \ps_0 \|_{1,w}^2 + \|  \ps _0\|^2 + \|  \vp_0\|_{1,w}^2) + C\int_0^t \int \fn{ |\na^3\vp|^2}. \label{claim2_lemma23_}
\end{align}
 As we did in Lemma \ref{lemma1_1} and
 \ref{lemma1_2}, our plan is to prove \eqref{claim1_lemma23_} first and to use the result in order to get \eqref{claim2_lemma23_}. Then we will close the estimate \eqref{beforeclaims_lemma23_} by using them. \\
 
 $\bullet$ Proof of \eqref{claim1_lemma23_}\\
 
Taking $ \dv$ to $\vp$ equation, we have
\begin{align*}
&\dv \vp_t - s \dv \vp_z - \Del \dv \vp \\
& = \quad N \Del \ps + \underbrace{ \na N \cdot \na\ps  +  \calP' \dv \vp + \calP ( \dv \vp)_z  
 +  \dv (\dv \vp \na\ps)}_{R_1 }.
\end{align*}
We multiply  $\Del \ps$ on the both sides 
to get
\begin{align*}
N |\Del \ps|^2 &=  ( \dv \vp_t - s\dv \vp_z - \Del \dv \vp) \Del \ps - \mbox{ R}_1\Del\ps\\
&= (\dv \vp \Del \ps )_t - \dv \vp \Del \ps_t - s \dv \vp_z \Del \ps -\underbrace{ \Del \dv \vp \Del \ps}_{(*)} -\mbox{ R}_1\Del\ps.
\end{align*}
For the second term $\dv \vp \Del \ps_t$, we use the $\psi$ equation (after taking $\Del$):
\[ \Del  \ps_t =  s\Del \ps_z  +\ep \Del\Del \ps +  \Del\lr{-2\ep  P \cdot\na \ps -\ep  |\na \ps|^2 }+ 
\Del  (\dv \vp)\] 
in order to get 
\begin{align*}
N |\Del \ps|^2  
&= (\dv \vp \Del \ps )_t - \dv \vp( s \Del \ps_z + \Del \dv \vp
) - s \dv \vp_z \Del \ps -(*)\\
&\quad-\mbox{ R}_1\Del\ps-\ep\dv \vp( 
 \Del\underbrace{\lr{\Del \ps -2   P \cdot\na \ps -   |\na \ps|^2 }}_{R_2}).
\end{align*}


 For $(*)$, we get
\begin{align*} \int(*)=&\int  \Del (\dv \vp) \Del \ps =\int  \Del (\ps_t - s \ps_z ) \Del \ps  -
 \int  \Del (\ep \Del \ps-2\ep  P \cdot\na \ps -\ep  |\na \ps|^2) \Del \ps\\
 & =\frac 12 \ddt \int |\Del \ps|^2  -
 \ep\int  \Del ( \Del \ps-2   P \cdot\na \ps -   |\na \ps|^2) \Del \ps
 =\frac 12 \ddt \int |\Del \ps|^2  -
 \ep\int  \Del R_2 \Del \ps.\end{align*}
 So, integrating on the strip, we have
\begin{align*}   \int N |\Del \ps|^2 &= \ddt \int \dv \vp \Del \ps + \int |\na  \dv \vp|^2 - \frac 12 \ddt \int |\Del \ps|^2 \\
 &  \quad - \int \mbox{ R}_1\Del\ps + \ep\int  \Del (  \mbox{R}_2 ) \Del \ps -\ep \int \dv \vp( 
 \Del\mbox{R}_2).\end{align*}
 
 We observe that 
 \begin{align*}
 \int_0^t\lr{\ddt \int \dv\vp \Del \ps - \frac 12 \ddt \int  |\Del\ps|^2}\leq C(\|\na\vp(t)\|^2 
 +\|\Del \ps_0\|^2 + \|\na \vp_0\|^2).&
 \end{align*} 
 
 The terms in $\int R_1\Del\psi$ are estimated as follows;
 \begin{align*}
& \int  | ( \na N \cdot \na\ps) \Del \ps |\le C\int \fn{ |\na\ps|^2 } + \frac 14 \int  N |\Del \ps|^2, \\
& \int |\calP' \dv \vp \Del \ps| +\int | \calP ( \dv \vp_z) \Del \ps| \le 
C \lr{ \int \fn{ |\na \vp|^2} + \int \fn{|\na^2 \vp|^2}} + \frac 18 \int N |\Del \ps|^2,\\
&\int |\dv (\dv\vp \na \ps ) \Del \ps | \le C ( \| \na \ps\|_{L^{\infty}} + \| \na \vp \|_{L^{\infty}})
 \lr{ C \int \fn{ |\na^2 \vp|^2} + C\int \fn{ |\Del \ps|^2} + \frac 14 \int N |\Del \ps|^2}\\
&\quad \qquad\qquad\qquad\qquad\leq C  
  \int \fn{ |\na^2 \vp|^2} + C\sqrt{M(t)}\int \fn{ |\Del \ps|^2} + \frac 14 \int N |\Del \ps|^2
 \end{align*} by assuming $\delta_0$ small enough.

For the $\ep$ terms,
we estimate them by
  \begin{align*} 
 & \ep\int \lr{ \Del (  \mbox{R}_2 ) \Del \ps -  \dv \vp( 
 \Del\mbox{R}_2)} \leq -\ep\int|\na\Del\ps|^2\\
 &\quad\qquad\qquad + \ep\int   \Del (  -2   P \cdot\na \ps -   |\na \ps|^2 ) \Del \ps - \ep\int  \dv \vp( 
 \Del\lr{\Del \ps -2   P \cdot\na \ps -   |\na \ps|^2 }) .
  \end{align*} For the second term and the third term, we estimate
  \begin{align*}
 &-\ep \int \Del(     2   P \cdot\na \ps+  |\na \ps|^2)\Del\ps
 =\ep \int \na(     2   P \cdot\na \ps+  |\na \ps|^2)\na\Del\ps\\
&\leq C\ep  \lr{\|\na\ps\|^2+\|\na^2\ps\|^2}+\frac{\ep}{4}\|\na \Del\ps\|^2\\
\end{align*} and 
\begin{align*}
&- \ep\int\dv\vp( \Del\lr{  \Del \ps  -2 ( P \cdot\na \ps) -  |\na \ps|^2 })
=\ep\int\na\dv\vp( \na\lr{  \Del \ps  -2 ( P \cdot\na \ps) -  |\na \ps|^2 })\\
&\leq  C\ep\int|\na^2\vp|(|\na\Del\ps|  +  |\na \ps| + |\na^2 \ps|+ |\na^2 \ps||\na\ps|)\\
&\leq \frac\ep 4\|\na\Del\ps\|^2 +C\ep  \|\na^2\vp\|^2+C\ep  \|\na\ps\|^2+C\ep  \|\na^2\ps\|^2.\\
\end{align*}
As a result, we get\begin{align*} 
 &\ep\int \lr{ \Del (  \mbox{R}_2 ) \Del \ps -  \dv \vp( 
 \Del\mbox{R}_2)} \leq -\frac \ep 2 \|\na\Del\ps\|^2 +C\ep  \|\na^2\vp\|^2+C\ep  \|\na\ps\|^2+C\ep  \|\na^2\ps\|^2.
  \end{align*}
Collecting the above estimates and using  Lemma \ref{lemma01_}, we have
 \begin{align*}\int_0^t \int N |\Del \ps|^2 +\ep\int_0^t \int  |\na\Del\ps|^2 \le &  C  ( \|  \vp_0\|_{1,w}^2  +
 \| \na \ps_0\|_w^2+ \| \ps_0\|^2+\|\Del\ps_0\|^2) \\
& + C\sqrt{M(t)} \int_0^t \int \fn{|\Del \ps|^2}.\end{align*}

To get \eqref{claim1_lemma23_} from the above estimate, we have to control 
$ \int N |\na^2\ps|^2$ by $\int N |\Del \ps|^2$ (possibly with lower order terms). 
Observe 
\begin{align*}&\int N |\Del \ps|^2 = \int N( (\pa_{zz}\ps)^2 +  (\pa_{yy}\ps)^2+  2\pa_{zz}\ps \pa_{yy}\ps)\\
&= \underbrace{\int N( (\pa_{zz}\ps)^2 +  (\pa_{yy}\ps)^2+  2(\pa_{zy}\ps)^2)}_{\int N |\na^2\ps|^2}
-2 \int N' \pa_z\ps \pa_{yy} \ps
\end{align*} 
 
and  
\[\Big| \int N' \pa_z \ps \pa_{yy} \ps \Big|= \Big|-\int(P+s) N \pa_z\ps \pa_{yy}\ps\Big|
\le \frac 14 \int N( \pa_{yy}\ps)^2
+ C\int N |\na\ps|^2.  \]  
Thus  we get
\begin{align} \label{la_1} \int N |\na^2\ps|^2
&\leq 2 \int N( (\pa_{zz}\ps)^2 + \frac 12 (\pa_{yy}\ps)^2+  2(\pa_{zy}\ps)^2)\\
&= 2\int N |\Del \ps|^2+4\int N' \pa_z \ps \pa_{yy} \ps -  \int N( \pa_{yy}\ps)^2 \nonumber\\
&\leq 2\int N |\Del \ps|^2  +  \int N( \pa_{yy}\ps)^2
+ C\int N |\na\ps|^2 -  \int N( \pa_{yy}\ps)^2\nonumber \\
&\leq 2\int N |\Del \ps|^2   
+ C\int N |\na\ps|^2.\nonumber   
  \end{align} 
 So we have
 \[ \int_0^t\int N |\na^2\ps|^2  \le 
   C ( \|  \vp_0\|_{1,w}^2  +
 \| \na \ps_0\|_w^2+ \| \ps_0\|^2)+ C\int_0^t\int N |\Del \ps|^2 \] 
by Lemma \ref{lemma01_}. Thus we proved \eqref{claim1_lemma23_}.\\

$\bullet$ Proof of \eqref{claim2_lemma23_}\\

Multiplying $ w \na^2 \ps$ to the equation $$ \na^2 \ps_t - s \na^2\ps_z -\ep \Del \na^2\ps = \na^2 (\dv \vp)
 -2\ep \na^2( P \cdot\na \ps) -\ep  \na^2(|\na \ps|^2), 
$$ we have
\begin{align*} 
&\frac 12 (w |\na^2 \ps|^2)_t - \frac s2 ( w |\na^2 \ps|^2)_z + \frac s2 w' |\na^2 \ps|^2 \\
&= w \na^2 (\dv \vp) \cdot \na^2 \ps 
+\underbrace{\ep w \na^2 \ps\cdot\lr{\Del \na^2\ps -2  \na^2( P \cdot\na \ps) -   \na^2(|\na \ps|^2)}}_{\ep- \mbox{terms}}.
\end{align*}
 
Recall that there exists a point $z_0\in\bbR$  such that
  \ben\begin{split}
&\frac{w'(z)}{w(z)}\geq \frac{s}{2}\quad \mbox{for } z\geq z_0\quad \mbox{and}\quad w(z)\leq \frac{4}{s^2}\leq \frac{16}{s^4} N \quad \mbox{for } z\leq z_0.
 \end{split}\een
 by  \eqref{rem_center}.

Integrating on each half strip 
(notation : $\int_{z>z_0}f:=\int_{z_0}^\infty\int_0^\la f (z,y,t)dy dz$)
and in time,  we get
\begin{align*}
\frac{1}{2}\int _{z>{z_0}} w |\na^2 \ps|^2 &
 \le  \frac{1}{2}\int_{z>{z_0}} w |\na^2 \ps_0|^2 + \int_0^t \int_{z>{z_0}} w \na^2 (\dv \vp) \cdot \na^2 \ps -
\frac{s^2}{4}\int_0^t \int _{z>{z_0}} w |\na^2 \ps|^2 \\&\quad -\frac s2\int_0^t  \int_0^\la w |\na^2 \ps|^2({z_0}, y) dy+\int_0^t \int_{z>{z_0}}{\ep\mbox{-terms}}\\
& \le \frac{1}{2} \int _{z>{z_0}} w |\na^2 \ps_0|^2 -\frac{s^2}{8} \int_0^t \int_{z>{z_0}} w |\na^2 \ps|^2 +
  C \int_0^t \int_{z>{z_0}} w |\na^2(\na \cdot \vp)|^2 \\&\quad -\frac s2\int_0^t  \int_0^\la w |\na^2 \ps|^2({z_0}, y) dy
  +\int_0^t \int_{z>{z_0}}{\ep\mbox{-terms}} 
  \end{align*} and
  \begin{align*}
\frac{1}{2} \int _{z<{z_0}} w |\na^2 \ps|^2 \le & \frac{1}{2} \int_{z<{z_0}} w |\na^2 \ps_0|^2 + \int_0^t \int_{z<{z_0}} w \na^2 (\dv \vp) \cdot \na^2 \ps \\&\quad
 +\frac s2\int_0^t  \int_0^\la w |\na^2 \ps|^2({z_0}, y) dy
 +\int_0^t \int_{z<{z_0}}{\ep\mbox{-terms}}\\
&\le \frac{1}{2}\int_{z<{z_0}} w |\na^2 \ps_0|^2 + C\int_0^t \int_{z<{z_0}}  | \na^3 \vp|||\na^2 \ps|  \\&\quad
  +\frac s2\int_0^t  \int_0^\la w |\na^2 \ps|^2({z_0}, y) dy
   +\int_0^t \int_{z<{z_0}}{\ep\mbox{-terms}}\\
  &\le  \frac{1}{2}\int_{z<{z_0}} w |\na^2 \ps_0|^2 + C\int_0^t \int_{z<{z_0}}  N|\na^2 \ps|^2  
  + C\int_0^t \int_{z<{z_0}}  { \fn{|\na^3 \vp|^2}} \\&\quad
  +\frac s2\int_0^t  \int_0^\la w |\na^2 \ps|^2({z_0}, y) dy
   +\int_0^t \int_{z<{z_0}}{\ep\mbox{-terms}}.
\end{align*}
As in  the proof of Lemma \ref{lemma1_2}, we get
\begin{align*}
&\frac{1}{2}\int w |\na^2 \ps|^2 + \frac{s^2}{8} \int_0^t \int w |\na^2 \ps|^2 \le \frac{1}{2}\int w |\na^2 \ps_0|^2
+ C\int_0^t \int  \Big(N  |\na^2\ps|^2 + 
  w |\na^3 \vp|^2  +  {\ep\mbox{-terms}}\Big).
\end{align*}

For the $\ep${-terms}, as before, we estimate

\begin{align*}
&\int {\ep\mbox{-terms}}=\ep\int { w \na^2 \ps\cdot\lr{\Del \na^2\ps -2  \na^2( P \cdot\na \ps) -   \na^2(|\na \ps|^2)}}\\
&=\ep\int { \Big(-w|\na^3\ps|^2 -w' \na^2 \ps\cdot  \na^2\ps_z
- w \na^2 \ps\cdot\lr{  2  \na^2( P \cdot\na \ps)} 
+  
(w'\na\psi_z+w\Del\na\ps)
\na(|\na \ps|^2)}\Big)\\
&\leq-\ep\int   w|\na^3\ps|^2 \\&\,\, +C\ep\int \Big(w |\na^2 \ps|| \na^3\ps| + w| \na^2 \ps|( |\na^3\ps|+|\na^2\ps|+|\na\ps| ) 
+w|\na \ps|( |\na^2 \ps|    |\na^2 \ps| +|\na^3\ps| |\na^2 \ps|)\Big)\\
&\leq-\frac \ep 4\int   w|\na^3\ps|^2 +C\ep\int w |\na^2 \ps|^2 +C\ep\int w |\na \ps|^2 
\end{align*} where we used the estimate $|w'|\leq C|w|$ and for the last inequality, we assumed 
$\delta_0$ small enough.\\

Collecting the above estimates, and using Lemma \ref{lemma01_} and the previous claim \eqref{claim1_lemma23_}, we have 
\begin{align*}
&\frac{1}{2}\int w |\na^2 \ps|^2 +   (\frac{s^2}{8}-C(\ep_0+\sqrt{\delta_0})) \int_0^t \int w |\na^2 \ps|^2 + \frac\ep4 \int_0^t\int   w|\na^3\ps|^2 \\
&\le C ( \| \na \ps_0\|_{1,w}^2 + \| \ps_0\|^2 + \| \vp_0\|_{1,w}^2) 
+ C \int_0^t \int \fn{ |\na^3 \vp|^2} . 
\end{align*}

  Then, by making $\ep_0$ and $\delta_0$ small enough, it proves the claim \eqref{claim2_lemma23_}.\\

Now we are ready to finish this proof for Lemma \ref{lemma23_} for the second order (the case $k=2$).
Plugging \eqref{claim2_lemma23_}  into \eqref{beforeclaims_lemma23_} with   small $\delta_0$, we  have
\begin{align*}
 &\| \na^2\vp\|_w^2 + \| \na^2\ps\|^2 + \int_0^t \| \na^3 \vp\|_w^2  + \ep \int_0^t\|\na^3\ps\|^2 \\
 &\le  C ( \| \na \ps_0\|_{1,w}^2 + \| \ps_0\|^2 + \| \vp_0\|_{2,w}^2). 
\end{align*}
In turns,  we have
\begin{align*}
\int \fn{ |\na^2 \ps|^2} + \int_0^t \int \fn{ |\na^2 \ps|^2} + 
 \ep\int_0^t \int \fn{ |\na^3 \ps|^2} \le C ( \|  \vp_0\|_{2,w}^2   + \| \na \ps_0\|_{1,w}^2
+ \| \ps_0\|^2).
\end{align*}   This proves Lemma \ref{lemma23_} for case $k=i+j=2$.

 \begin{remark}

Together with Lemma \ref{lemma01_}, we have proved 
\begin{align}\label{lemma2_} &
\|  \vp\|_{2,w}^2  +
 \| \na \ps\|_{1,w}^2+ \| \ps\|^2
+  \int_0^t  \sum_{l = 1,2,3} \| \na^{l} \vp\|_w^2
+\int_0^t   \sum_{l = 1,2}\| \na^l \ps\|_w^2
+\ep\int_0^t  \sum_{l = 1,2,3} \| \na^{l} \ps\|_w^2\\
 &\le C( \|  \vp_0\|_{2,w}^2  +
 \| \na \ps_0\|_{1,w}^2+ \| \ps_0\|^2). \nonumber
\end{align}
  \end{remark}
 \ \\

$\bullet$ Case $ k=i + j =3$\\
\ \\
For $k=3$, we present its proof for completeness even if  there is almost no new idea.  First, we recall the equation \eqref{higher_}. As before, we estimate
\begin{align*}
&\frac 12 \int |\na^3\vp|^2 \lr{ \fn{1}}''-\frac s 2 \int |\na^3\varphi|^2 \left( \fn{1}\right)'
\le  \frac{1}{8}\int  |\na^4\vp |^2\left( \fn{1}\right)+C \int |\na^3\varphi|^2  \left( \fn{1}\right). 
 \end{align*}
Observe that the quadratic terms are symbolically
\begin{align*}
 \fn{N^{(l)}} \na\na^{3-l} \ps \na^3 \vp  \quad \mbox{ for } l=1,2,3 \quad \mbox{ and }\\
 \fn{P^{(l)}} \dv(\na^{3-l} \vp) \na^3 \vp \quad \mbox{ for } l=0,1,2,3.\\
\end{align*}

 After integration by parts, the terms with $l=3$ are bounded by
\begin{align*}
   \Big|\int \fn{N'''}\na\ps\cdot\vp_{zzz} \Big|
&\leq C\int \Big[ \Big|\fn{N''}\Big||\na\ps||\na^3\vp|+\Big|\fn{N''}\Big||\na^2\ps||\na^3\vp|+\Big|\fn{N''}\Big||\na\ps||\na^4\vp|\Big]\\
&\leq C\Big( 
\|\fnn{\na \ps}\|^2+\|\fnn{\na^3 \vp}\|^2+\|\fnn{\na^2 \ps}\|^2
\Big)+\frac{1}{8}\|\fnn{\na^4 \vp}\|^2 ,
 \end{align*} 
\begin{align*}
   \Big|\int \fn{\PP'''}(\dv\vp)(\vp^1)_{zzz}\Big|
&\leq C \int \Big[ \Big|\fn{\PP''}\Big||\na \vp||\na^3\vp|
+\Big|\fn{\PP''}\Big||\na^2 \vp||\na^3\vp|+\Big|\fn{\PP''}\Big||\na \vp||\na^4\vp|\Big]\\
&\leq C\Big( 
\|\fnn{\na \vp}\|^2+\|\fnn{\na^3 \vp}\|^2+\|\fnn{\na^2 \vp}\|^2
\Big)+\frac{1}{8}\|\fnn{\na^4 \vp}\|^2.
 \end{align*} 
 
All the other quadratic  terms   are estimated by 
\begin{align*}
&  C\int 
\Big[ 
\Big|\fn{N''}\Big|
| \na^2 \ps|| \na^3 \vp|
+ \Big|\fn{P''} \Big||\na^2 \vp ||\na^3 \vp|
 +\Big|\fn{P'} \Big||\na^3 \vp ||\na^3 \vp|\Big]\\
 &\leq  C \lr{
 \| \frac{\na^2\ps }{\sqrt{N}}\|^2
 +\| \frac{\na\vp }{\sqrt{N}}\|^2+\| \frac{\na^2\vp }{\sqrt{N}}\|^2 +\| \frac{\na^3\vp }{\sqrt{N}}\|^2},
\end{align*}

\begin{align*}
C\int \Big|\frac{P}{N}\Big|  |\na^4 \vp| |\na^3\vp|   
 \le
 \frac{1}{8} \|  \frac{\na^4\vp }{ \sqrt{N}} \|^2 +C\| \frac{\na^3\vp }{\sqrt{N}}\|^2
\end{align*}  and 
\begin{align*}
& C\Big|\int \fn{N'} \na^3 \ps \na^3 \vp\Big|\leq C \int  | \na^2 \ps| |\na^3 \vp|
+C\int   |\na^2 \ps ||\na^4 \vp|\\
 &\leq  C \|  {\na^2\ps } \|\lr{\|  {\na^3\vp } \|+\|  {\na^4\vp } \|}\leq 
 C \|  {\na^2\ps } \|^2+ C\|  {\na^3\vp } \|^2+\frac 1 8 \|  {\na^4\vp } \|^2
\end{align*} where we used integration by parts for the last estimate.

So by \eqref{lemma2_}, 
we have
\begin{align*} &\int \fn{ |\na^3 \vp|^2} +\int |\na^3 \ps|^2 + \int_0^t \int \fn{ |\na^4 \vp|^2} +
 \ep\int_0^t \int { |\na^4 \ps|^2}\\
 &\le C( \|  \vp_0\|_{3,w}^2  + \| \na^3 \ps_0\|^2+
 \| \na \ps_0\|_{1,w}^2+ \| \ps_0\|^2) \\
 & +\mbox{the cubic terms } +\mbox{the }\ep\mbox{-terms }.
 \end{align*}

The cubic terms are estimated by 
\begin{align*} C\left|\int \na^3 (\na \vp \na \ps) \fn{ \na^3 \vp} \right|
&\le C\left|  \int  \na^2 ( \na \vp \na \ps) \lr{ \fn{\na^4 \vp} +  {\na^3 \vp}(\frac{1}{N})'} \right |\\
&\le C\left|  \int  \na^2 ( \na \vp \na \ps)   \fn{\na^4 \vp}   \right |
+C\left|  \int  \na^2 ( \na \vp \na \ps)   {\na^3 \vp}(\frac{1}{N})'  \right |.
 \end{align*}
From  $\left|(\frac{1}{N})'\right|+\left|(\frac{1}{N})''\right|\leq \frac{C}{N}$ (Lemma \ref{lem_NP}), we can estimate each term. Indeed, recall $$\|\na\ps\|_{L^\infty}+
\|\na\vp\|_{L^\infty}\leq C\sqrt{M(t)}.$$ So we get
\begin{align*}
&C\left| \int \na^3 \vp \na \ps \fn{ \na^4\vp}\right| + C\left|\int \na \vp \na^3 \ps \fn{ \na^4 \vp}\right|
\le  C\sqrt{M(t)} \lr{ \int \fn{|\na^3\vp|^2} +   \fn{|\na^3 \ps|^2} }+\frac{1}{8}\int \fn{|\na^4 \vp|^2} ,\end{align*}
\begin{align*}
C\left|\int \na^3 \vp \na \ps  { \na^3\vp} (\frac{1}{N})' \right| + \left|\int\na \vp \na^3 \ps { \na^3 \vp}(\frac{1}{N})' \right|
&\le C \int |\na^3 \vp ||\na \ps |\fn{| \na^3\vp|}  + |\na \vp|| \na^3 \ps| \fn{ |\na^3 \vp|} \\
&\le  C\sqrt{M(t)} \lr{ \int \fn{|\na^3\vp|^2} +   \fn{|\na^3 \ps|^2} }\quad\mbox{and}
\end{align*}
\begin{align*}
& C\left|\int \na^2  \vp \na^2 \ps  {\na^3 \vp} (\frac{1}{N})' \right|\\&\,\,
\leq C\left|\int \na^3  \vp \na \ps  {\na^3 \vp} (\frac{1}{N})'\right| +C\left|
\int \na^2  \vp \na  \ps  {\na^4 \vp} (\frac{1}{N})'  \right|
+C\left|\int \na^2  \vp \na  \ps  {\na^4 \vp} (\frac{1}{N})''\right| \\
&\,\,\le C   \int \Big[|\na^3 \vp|| \na \ps |\fn{|\na^3\vp|} +  |\na^2 \vp ||\na\ps| \fn{|\na^4 \vp|} \Big]
\le C \sqrt{M(t)} \lr{  
 \int \fn{|\na^3\vp|^2} +   \fn{|\na^2 \vp|^2}
} +  \frac{1}{8} \int \fn{|\na^4\vp|^2}.
\end{align*}
Then the  term $ \int \na^2\vp \na^2 \ps \fn{\na^4 \vp}$ remains. Note that by the Sobolev embedding, 
$$\|f\|_{L^4}\leq C(\|f\|_{L^2}+\|\na f\|_{L^2}).$$
So we  estimate
\begin{align*}
&C\Big|\int \na^2\vp \na^2 \ps \fn{\na^4 \vp}\Big|  
\le C \|    {\na^2 \ps} \|_{L^4}
\|  \frac {\na^2 \vp}{\sqrt  N}\|_{L^4}\| \|  \frac {\na^4 \vp}{\sqrt N}\|_{L^2} \\
& \le \frac 18 \|  \frac {\na^4 \vp}{\sqrt N}\|^2_{L^2} + 
C    \lr{ \|  \na\left( \frac {\na^2 \vp}{\sqrt  N}\right)\|_{L^2} +\|  \frac {\na^2 \vp}{\sqrt  N}\|_{L^2}}^2 
\cdot \lr{ \|      {\na^3 \ps}  \|_{L^2} +\|    {\na^2 \ps} \|_{L^2}}^2\\
& \le \frac 18 \|  \frac {\na^4 \vp}{\sqrt N}\|^2_{L^2} + 
C    \lr{ \|    \frac {|\na^3 \vp|}{\sqrt N}+|\na^2\vp|
{|(\frac{1}{\sqrt{N}})'|}\|_{L^2} +C\|  \frac {\na^2 \vp}{\sqrt N}\|_{L^2}}^2 
\cdot \lr{ \|      {\na^3 \ps}  \|_{L^2} +{\|    {\na^2 \ps} \|_{L^2}}}^2.\\
\end{align*}
Using ${|(\frac{1}{\sqrt{N}})'|}{\leq \frac{C}{\sqrt{N}}}$ (Lemma \ref{lem_NP}), we get
\begin{align*} & \lr{ \|    \frac {|\na^3 \vp|}{\sqrt N}+|\na^2\vp|
 {|(\frac{1}{\sqrt{N}})'|} \|_{L^2}
+C\|  \frac {\na^2 \vp}{\sqrt N}\|_{L^2}}^2 
\leq 
C \lr{\|    \frac {\na^3 \vp}{\sqrt N}  \|^2_{L^2} +\|  \frac {\na^2 \vp}{\sqrt N}\|_{L^2}^2}
\leq CM(t).
\end{align*}

 For the $\ep$-terms, we can write them  symbolically:
$$ 
    \ep \int \na^3(\calP \ps_z) \na^3\psi  \quad\mbox{ and }\quad
 \ep \int\nabla^3(|\na \ps|^2) \nabla^3\ps.
$$
After integration by parts, we can estimate these terms by
\begin{align*}
 & C\ep \int |\na^2(\calP \ps_z)| |\na^4\psi|\leq
 C\ep \int (|\na \ps| +|\na^2 \ps|+|\na^3 \ps|) |\na^4\psi|\\
 &\leq  C\ep ( \|\nabla\ps\|^2+ \|\nabla^2\ps\|^2+\|\nabla^3\ps\|^2) +\frac{\ep}{4}\|\nabla^4\ps\|^2\\
\end{align*}
 and
\begin{align*}&C\ep \int|\na^2(|\na \ps|^2)| |\na^4\ps|
\leq C\ep\int(|\nabla\ps||\nabla^3\ps|+|\nabla^2\ps||\nabla^2\ps|) |\nabla^4\ps|\\
&\leq C\ep \sqrt{M(t)}\int  |\nabla^3\ps| |\nabla^4\ps|+C\ep \|\na^2\ps\|^2_{L^4}\cdot\|\na^4\ps\|_{L^2} \\
&\leq  C\ep \sqrt{M(t)} \|\nabla^3\ps\|^2 +\frac{\ep}{8}\|\nabla^4\ps\|^2+C\ep\underbrace{(\|\na^3\ps\|^2_{L^2}
+\|\na^2\ps\|^2_{L^2})}_{\leq C M(t)}\cdot\|\na^4\ps\|_{L^2}\\
&\leq  C\ep \sqrt{M(t)} \|\nabla^3\ps\|^2 +\frac{\ep}{4}\|\nabla^4\ps\|^2
\end{align*} by assuming $\delta_0$ small enough.

 Collecting the above estimates and using \eqref{lemma2_}, we get
 the third order version of \eqref{beforeclaims_lemma23_}:
\begin{align}\label{beforeclaims34_lemma23_} &\int \fn{ |\na^3 \vp|^2} +\int |\na^3 \ps|^2 + \int_0^t \int \fn{ |\na^4 \vp|^2} +
 \ep\int_0^t \int{ |\na^4 \ps|^2}\\
 &\le C( \|  \vp_0\|_{3,w}^2  + \| \na^3 \ps_0\|^2+
 \| \na \ps_0\|_{1,w}^2+ \| \ps_0\|^2) 
 +C\sqrt{M(t)}\int_0^t \int  \fn{|\na^3 \ps|^2}.   \nonumber
 \end{align}  
 
 As before, we claim the following two estimates for the third order derivatives:
 \begin{align} \label{claim3_lemma23_}
&\int_0^t \int N |  \na^3 \ps|^2  +\ep\int_0^t \int  |\na^2\Del\ps|^2 \\&\le   C  ( \|  \vp_0\|_{2,w}^2  +
 \| \na \ps_0\|_{1,w}^2+ \| \ps_0\|^2+\|\na\Del\ps_0\|^2) 
+ C\sqrt{M(t)} \int_0^t \int \fn{|\na^3 \ps|^2},\nonumber\\
& \int \fn {|\na^3  \ps|^2} +   \int _0^ t \int \fn {|\na^3 \ps|^2 } 
+ \ep\int _0^ t \int \fn {|\na^4 \ps|^2 } \le  C    ( \| \na \ps_0 \|_{2,w}^2 + \|  \ps _0\|^2 + \|  \vp_0\|_{2,w}^2)
+ C\int_0^t \int \fn{ |\na^4\vp|^2}. \label{claim4_lemma23_}
\end{align}
We will prove \eqref{claim3_lemma23_} below and we will use the result in order to get \eqref{claim4_lemma23_}. Then we will apply \eqref{claim4_lemma23_} to close \eqref{beforeclaims34_lemma23_}. \\

$\bullet$ Proof of \eqref{claim3_lemma23_}\\

Taking $ D\dv$ to $\vp$ equation where $D$ is either $\partial_z$ or $\partial_y$, we have
\begin{align*}
&D\dv \vp_t - s D\dv \vp_z - \Del D\dv \vp \\
& = \quad N D\Del \ps + \underbrace{D N\Del\ps+D\Big(\na N \cdot\na\ps  + \dv P \dv \vp
+ P\cdot\na ( \dv \vp) + \dv ( \dv \vp \na \ps)\Big)}_{R_1 }.
\end{align*}

We multiply  $D\Del \ps$ on the both sides
to get
\begin{align*}
&N |D\Del \ps|^2 =  D( \dv \vp_t - s\dv \vp_z - \Del \dv \vp) D\Del \ps - \mbox{ R}_1D\Del\ps\\
&= (D\dv \vp D\Del \ps )_t - D\dv \vp D\Del \ps_t - s D\dv \vp_z D\Del \ps -\underbrace{ D\Del \dv \vp D\Del \ps}_{(*)} -\mbox{ R}_1D\Del\ps.
\end{align*}

For the second term $D\dv \vp D\Del \ps_t$, we use the $\psi$ equation (after taking $D\Del$):
\[ D\Del  \ps_t  =  sD\Del \ps_z  +\ep \Del D\Del \ps +  D\Del\lr{-2\ep  P \cdot\na \ps -\ep  |\na \ps|^2 }+ 
D\Del  (\dv \vp)\]
in order to get
\begin{align*}
 N |D\Del \ps|^2  
&= (D\dv \vp D\Del \ps )_t - D\dv \vp( s D\Del \ps_z + D\Del \dv \vp
) - s D\dv \vp_z D\Del \ps - (*)\\
&\quad-\mbox{ R}_1D\Del\ps -\ep D\dv \vp( 
 D\Del\underbrace{\lr{\Del \ps -2   P \cdot\na \ps -   |\na \ps|^2 }}_{R_2}).
\end{align*}
 We observe that
\begin{align*} \int(*)&=\int  D\Del (\dv \vp) D\Del \ps =\int  D\Del (\ps_t - s \ps_z ) D\Del \ps  -
 \int  D\Del (\ep \Del \ps-2\ep  P \cdot\na \ps -\ep  |\na \ps|^2) D\Del \ps\\
 & =\frac 12 \ddt \int |D\Del \ps|^2  -
 \ep\int  D\Del ( \Del \ps-2   P \cdot\na \ps -   |\na \ps|^2) D\Del \ps\\
&=\frac 12 \ddt \int |D\Del \ps|^2  -
 \ep\int  D\Del ( R_2) D\Del \ps .
 \end{align*}
 
  So, integrating on the strip, we have
\begin{align*}   \int N |D\Del \ps|^2 =& \ddt \int D\dv \vp D\Del \ps + \int |\na D  \dv \vp|^2 - \frac 12 \ddt \int |D\Del \ps|^2 \\
 &  - \int \mbox{ R}_1D\Del\ps  + \ep\int D \Del (  \mbox{R}_2 ) D\Del \ps -\ep \int D\dv \vp( 
 D\Del\mbox{R}_2).\end{align*}

  Note that 
 \begin{align*}
 \int_0^t\lr{\ddt \int D\dv\vp D\Del \ps - \frac 12 \ddt \int  |D\Del\ps|^2}\leq C\Big(\|\na^2\vp(t)\|^2 
 + \|D\Del \ps_0\|^2+\|\na^2\vp_0\|^2 \Big)&.
 \end{align*}

  The integral containing $R_1$ is estimated as follows;
  \\ The quadratic terms:
 \begin{align*}
 &C\Big|\int (DN)(\Del\ps )(D\Del\ps)\Big|\leq  C \int  |\Del \ps|^2 + \frac 18 \int N |D\Del \ps|^2,\\
 &C\Big|\int   D( \na N \cdot \na\ps) D\Del \ps\Big| \le C\int \fn{ |\na\ps|^2 +|\na^2\ps|^2} + \frac 18 \int  N |D\Del \ps|^2 \quad\mbox{and} \\
&C \Big|\int D(\calP' \dv \vp) D\Del \ps + D(\calP ( \dv \vp_z))D \Del \ps\Big| \le 
C \lr{ \int \fn{ |\na \vp|^2+|\na^2 \vp|^2+|\na^3 \vp|^2}  } + \frac 18 \int N |D\Del \ps|^2. 
 \end{align*}  The cubic term:
 \begin{align*}
&   C\Big|\int D(\dv (\dv\vp \na \ps ) )D\Del \ps \Big|\\& \le C ( \| \na \ps\|_{L^{\infty}} + \| \na \vp \|_{L^{\infty}})
 \lr{ C \int \fn{ |\na^3 \vp|^2} + C\int \fn{ |D\Del \ps|^2} + \frac 14 \int N |D\Del \ps|^2}
 + C\int |\na^2\vp||\na^2\ps||D\Del\ps|\\
& \leq C  
  \int \fn{ |\na^3 \vp|^2} + C\sqrt{M(t)}\int \fn{ |D\Del \ps|^2} + \frac 14 \int N |D\Del \ps|^2
   + C\int |\na^2\vp||\na^2\ps||D\Del\ps| 
 \end{align*} 
  for small $\delta_0$. The last term in the above can be estimated:
 
 \begin{align*}
&C\int |\na^2\vp||\na^2\ps||D\Del\ps|=C\int \frac{|\na^2\vp|}{\sqrt{N}}{|\na^2\ps|}\sqrt{N}|D\Del\ps|
\le C \|    {\na^2 \ps} \|_{L^4}
\|   \frac{\na^2 \vp}{\sqrt{N}}\|_{L^4}\| \| \sqrt{N} D\Del\ps\|_{L^2} \\
& \le \frac 18 \| \sqrt{N} D\Del\ps\|^2_{L^2} + 
C    \lr{ \|  \na\left( \frac {\na^2 \vp}{\sqrt  N}\right)\|_{L^2} +\|  \frac {\na^2 \vp}{\sqrt  N}\|_{L^2}}^2 
\cdot \lr{ \|      {\na^3 \ps}  \|_{L^2} +\|    {\na^2 \ps} \|_{L^2}}^2\\
& \le \frac 18 \| \sqrt{N} D\Del\ps\|^2_{L^2} + 
C    \underbrace{\lr{ \|    \frac {|\na^3 \vp|}{\sqrt N}+|\na^2\vp|
{|(\frac{1}{\sqrt{N}})'|}\|_{L^2} +C\|  \frac {\na^2 \vp}{\sqrt N}\|_{L^2}}^2}_{\leq CM(t)}
\cdot \lr{ \|      {\na^3 \ps}  \|_{L^2} +{\|    {\na^2 \ps} \|_{L^2}}}^2.
\end{align*}

Up to now, using  \eqref{lemma2_}, we have
 \begin{align*}&\int_0^t \int N |D\Del \ps|^2  \le   C  ( \|  \vp_0\|_{2,w}^2  +
 \| \na \ps_0\|_{1,w}^2+ \| \ps_0\|^2+\|D\Del\ps_0\|^2) 
+ C\sqrt{M(t)} \int_0^t \int \fn{|\na^3 \ps|^2}\\ &
+ \ep\int \lr{ D\Del (  \mbox{R}_2 ) D\Del \ps -  D\dv \vp( 
 D\Del\mbox{R}_2)} \end{align*}
where $R_2={\lr{\Del \ps -2   P \cdot\na \ps -   |\na \ps|^2 }}.$\\

 For $\ep$ terms,
we estimate them by
  \begin{align*} 
 &\ep\int \lr{ D\Del (  \mbox{R}_2 ) D\Del \ps -  D\dv \vp( 
 D\Del\mbox{R}_2)} \leq -\ep\int|\na D\Del\ps|^2\\
 &+ \ep\int   D\Del (  -2   P \cdot\na \ps -   |\na \ps|^2 ) D\Del \ps - \ep\int D\dv \vp( 
 D\Del\lr{\Del \ps -2   P \cdot\na \ps -   |\na \ps|^2 }) .
  \end{align*} For the last two terms,
  thanks to $$ {\|\na^2\ps\|^4_{L^4}} {\leq C(\|\na^2\ps\|+\|\na^3\ps\|)^4
\leq CM(t)(\|\na^2\ps\|+\|\na^3\ps\|)^2
},$$ we estimate
  \begin{align*}
 &-\ep \int\Del D(     2   P \cdot\na \ps+  |\na \ps|^2)D\Del\ps
 =\ep \int \na D(     2   P \cdot\na \ps+  |\na \ps|^2)\na D\Del\ps\\
&\leq C\ep  \lr{\|\na\ps\|^2+\|\na^2\ps\|^2+\|\na^3\ps\|^2
+ {\|\na^2\ps\|^4_{L^4}} }+\frac{\ep}{4}\|\na D\Del\ps\|^2\\
&\leq C\ep  \lr{\|\na\ps\|^2+\|\na^2\ps\|^2+\|\na^3\ps\|^2}+\frac{\ep}{4}\|\na D\Del\ps\|^2\\
\end{align*} and 
\begin{align*}
&- \ep\int D\dv\vp(  \Del D\lr{  \Del \ps  -2 ( P \cdot\na \ps) -  |\na \ps|^2 })
=\ep\int\na D\dv\vp( \na D\lr{  \Del \ps  -2 ( P \cdot\na \ps) -  |\na \ps|^2 })\\
&\leq  C\ep\int|\na^3\vp|\lr{|\na D\Del\ps|  +  |\na \ps| + |\na^2 \ps|+ |\na^3 \ps|+ |\na^3 \ps||\na\ps|+  |\na^2\ps|^2}\\
&\leq \frac\ep 4\|\na D\Del\ps\|^2 +C\ep  \|\na^3\vp\|^2+C\ep  \|\na\ps\|^2+C\ep  \|\na^2\ps\|^2 +C\ep  \|\na^3\ps\|^2
+C\ep
{ \|\na^2\ps\|_{L^4}^4}\\
&\leq \frac\ep 4\|\na D\Del\ps\|^2 +C\ep  \|\na^3\vp\|^2+C\ep  \|\na\ps\|^2+C\ep  \|\na^2\ps\|^2 +C\ep  \|\na^3\ps\|^2.
\end{align*} 
As a result, we get \begin{align*} 
  \ep\int \lr{ D\Del (  \mbox{R}_2 ) D\Del \ps -  D\dv \vp( 
 D\Del\mbox{R}_2)} \leq & -\frac \ep 2 \|\na D\Del\ps\|^2 \\&+C\ep \lr{ \|\na^3\vp\|^2+   \|\na\ps\|^2
 +   \|\na^2\ps\|^2+   \|\na^3\ps\|^2}. 
  \end{align*}

We use  \eqref{lemma2_} again to get
 \begin{align*}&\int_0^t \int N |\Del D \ps|^2 +\ep\int_0^t \int  |\na D\Del\ps|^2 \\&\le   C  ( \|  \vp_0\|_{2,w}^2  +
 \| \na \ps_0\|_{1,w}^2+ \| \ps_0\|^2+\|D\Del\ps_0\|^2) 
+ C\sqrt{M(t)} \int_0^t \int \fn{|\na^3 \ps|^2}.\end{align*} 

   We replace $D$ with $\partial_z$ and $\partial_y$ and add these two estimates to get
 \begin{align}\label{claim3_middle}&\int_0^t \int N |\Del \na \ps|^2  +\ep\int_0^t \int  |\na^2\Del\ps|^2 \\&\le   C  ( \|  \vp_0\|_{2,w}^2  +
 \| \na \ps_0\|_{1,w}^2+ \| \ps_0\|^2+\|\na\Del\ps_0\|^2) 
+ C\sqrt{M(t)} \int_0^t \int \fn{|\na^3 \ps|^2}.\nonumber \end{align}
To get \eqref{claim3_lemma23_} from the above estimate, we have to estimate 
$ \int N |\na^3\ps|^2=\int N |\na^2 (\na\ps)|^2$ from $\int N |\Del (\na \ps)|^2$ (possibly with lower order terms). 
 We apply the estimate \eqref{la_1} by replacing $\ps$ with $\na \ps$ then we get
\begin{align*} \int N |\na^2(\na\ps)|^2
&\leq 2\int N |\Del (\na\ps)|^2   
+ C\int N |\na(\na\ps)|^2.  
  \end{align*}  
 So we have
 \[ \int_0^t\int N |\na^3\ps|^2  \le 
 C\int_0^t\int N |\Del \na\ps|^2  + C ( \|  \vp_0\|_{2,w}^2  +
 \| \na \ps_0\|_{1,w}^2+ \| \ps_0\|^2) \] 
by \eqref{lemma2_}. Together with \eqref{claim3_middle}, it proves \eqref{claim3_lemma23_}.\\

$\bullet$ Proof of \eqref{claim4_lemma23_}\\

Multiplying $ w \na^3\ps$ to the equation $$ \na^3 \ps_t - s \na^3\ps_z -\ep \Del \na^3\ps = \na^3 (\dv \vp)
 -2\ep \na^3( P \cdot\na \ps) -\ep  \na^3(|\na \ps|^2), 
$$ we have
\begin{align*} 
&\frac 12 (w |\na^3 \ps|^2)_t - \frac s2 ( w |\na^3 \ps|^2)_z + \frac s2 w' |\na^3 \ps|^2 \\
&= w \na^3 (\dv \vp) \cdot \na^3 \ps 
+\underbrace{\ep w \na^3 \ps\cdot\lr{\Del \na^3\ps -2  \na^3( P \cdot\na \ps) -   \na^3(|\na \ps|^2)}}_{ \ep\mbox{-terms}  }.
\end{align*}

As before, we use the point $z_0\in\bbR$ satisfying
  \ben\begin{split}
&\frac{w'(z)}{w(z)}\geq \frac{s}{2}\quad \mbox{for } z\geq z_0\quad \mbox{and}\quad w(z)\leq \frac{4}{s^2}\leq \frac{16}{s^4} N \quad \mbox{for } z\leq z_0.
 \end{split}\een
 by  \eqref{rem_center}.  
 
Integrating on each half strip    and in time, we get
\begin{align*}
\frac{1}{2}\int _{z>{z_0}} w |\na^3 \ps|^2 &
 \le \frac{1}{2} \int_{z>{z_0}} w |\na^3 \ps_0|^2 + \int_0^t \int_{z>{z_0}} w \na^3 (\dv \vp) \cdot \na^3 \ps -
\frac{s^2}{4}\int_0^t \int _{z>{z_0}} w |\na^3 \ps|^2 \\&\quad -\frac s2\int_0^t  \int_0^\la w |\na^3 \ps|^2({z_0}, y) dy
+\int_0^t \int_{z>{z_0}}{\ep\mbox{-terms}}\\
& \le  \frac{1}{2}\int _{z>{z_0}} w |\na^3 \ps_0|^2 -\frac{s^2}{8} \int_0^t \int_{z>{z_0}} w |\na^3 \ps|^2 +
  C \int_0^t \int_{z>{z_0}} w |\na^3(\na \cdot \vp)|^2 \\&\quad  -\frac s2\int_0^t  \int_0^\la w |\na^3 \ps|^2({z_0}, y) dy
  +\int_0^t \int_{z>{z_0}}{\ep\mbox{-terms}}
  \end{align*} and
  \begin{align*}
\frac{1}{2} \int _{z<{z_0}} w |\na^3 \ps|^2 \le & \frac{1}{2}\int_{z<{z_0}} w |\na^3 \ps_0|^2 + \int_0^t \int_{z<{z_0}} w \na^3 (\dv \vp) 
 \cdot \na^3 \ps \\&\quad 
 +\frac s2\int_0^t  \int_0^\la w |\na^3 \ps|^2({z_0}, y) dy
 +\int_0^t \int_{z<{z_0}}{\ep\mbox{-terms}}\\
&\le \frac{1}{2}\int_{z<{z_0}} w |\na^3 \ps_0|^2 + C\int_0^t \int_{z<{z_0}}  | \na^4 \vp|||\na^3 \ps| 
 \\&\quad  +\frac s2\int_0^t  \int_0^\la w |\na^3 \ps|^2({z_0}, y) dy
  +\int_0^t \int_{z<{z_0}}{\ep\mbox{-terms}}\\
  &\le \frac{1}{2} \int_{z<{z_0}} w |\na^3 \ps_0|^2 + C\int_0^t \int_{z<{z_0}}  N|\na^3 \ps|^2  
  + C\int_0^t \int_{z<{z_0}}  { \fn{|\na^4 \vp|^2}}
 \\&\quad  +\frac s2\int_0^t  \int_0^\la w |\na^3 \ps|^2({z_0}, y) dy
   +\int_0^t \int_{z<{z_0}}{\ep\mbox{-terms}}.
\end{align*}  
As before, we get
\begin{align*}
&\frac{1}{2} \int w |\na^3 \ps|^2 + \frac{s^2}{8} \int_0^t \int w |\na^3 \ps|^2 \le \frac{1}{2} \int w |\na^3 \ps_0|^2
+ C\int_0^t \int \Big(N  |\na^3\ps|^2 + 
  w |\na^4 \vp|^2  +  {\ep\mbox{-terms}}\Big). 
\end{align*} 
For $\ep${-terms}, as before, we estimate
\begin{align*}
&\int {\ep\mbox{-terms}}=\ep\int { w \na^3 \ps\cdot\lr{\Del \na^3\ps -2  \na^3( P \cdot\na \ps) -   \na^3(|\na \ps|^2)}}\\
&=\ep\int \Big({ -w|\na^4\ps|^2 -w' \na^3 \ps\cdot  \na^3\ps_z
- w \na^3 \ps\cdot\lr{  2  \na^3( P \cdot\na \ps)} 
+  
(w'\na^2\psi_z+w\Del\na^2\ps)
\na^2(|\na \ps|^2)}\Big)\\
&\leq-\ep\int   w|\na^4\ps|^2 +C\ep\int w\ps_{zzz}\cdot\PP'''\ps_z  +C\ep\int \Big[w |\na^3 \ps|| \na^4\ps| + w| \na^3 \ps|( |\na^4\ps|+|\na^3\ps|+|\na^2\ps| 
 )
 \\ &\quad
+w |\na^3 \ps|   \lr{ |\na^3 \ps||\na \ps|+ |\na^2 \ps||\na^2\ps|}+w|\na^4\ps| \lr{ |\na^3 \ps||\na \ps|+ |\na^2 \ps||\na^2\ps|}\Big]\\
&\leq-\frac \ep 4\int   w|\na^4\ps|^2 +C\ep\int w \sum_{k=1}^3 |\na^k \ps|^2
+C\ep\int \Big(w |\na^3 \ps|   |\na^2 \ps|^2+w|\na^4\ps|   |\na^2 \ps|^2\Big). 
\end{align*}  For the last  term (the cubic term), we estimate
\begin{align*}
&C\ep\int \Big(w |\na^3 \ps|   |\na^2 \ps|^2+w|\na^4\ps|   |\na^2 \ps|^2\Big)
\le    C\ep\lr{\|\sqrt{w}\na^3\ps\|+\|\sqrt{w}\na^4\ps\|}\|\sqrt{w}\na^2\ps\|_{L^4}\| \na^2\ps\|_{L^4}\\
&\le   C\ep \lr{\|\sqrt{w}\na^3\ps\|+\|\sqrt{w}\na^4\ps\|}\lr{\||\sqrt{w}\na^3\ps|+\underbrace{|(\sqrt{w})'|}_{
\leq C\sqrt{w}}
|\na^2\ps|\|+\|\sqrt{w}\na^2\ps\|}(\| \na^3\ps\|+\| \na^2\ps\|)\\
&\le   C \ep \lr{\|\sqrt{w}\na^3\ps\|+\|\sqrt{w}\na^4\ps\|}\cdot\sqrt{ M(t)}\cdot
(\| \na^3\ps\|+\| \na^2\ps\|)\\
& \le \frac \ep 8\int   w|\na^4\ps|^2 +C\ep\int  w\sum_{k=2}^3 |\na^k \ps|^2
\end{align*} for small $\delta_0$.\\

 In sum, using the estimate \eqref{lemma2_} and the previous claim \eqref{claim3_lemma23_}, we obtain\begin{align*}
&\frac{1}{2}\int w |\na^3 \ps|^2 +  (\frac{s^2}{8}-C(\ep_0+\sqrt{\delta_0})) \int_0^t \int w |\na^3 \ps|^2+ \frac\ep 4  \int   w|\na^4\ps|^2 \\
&\le C  (\| \na \ps_0\|_{2,w}^2+
  \| \ps_0\|^2 + \| \vp_0\|_{2,w}^2
)
+ C \int_0^t \int \fn{ |\na^4 \vp|^2}. 
\end{align*} By making $\ep_0$ and $\delta_0$ small enough, we proved the claim \eqref{claim4_lemma23_}.\\

Now we can prove Lemma \ref{lemma23_} fully. Indeed
plugging \eqref{claim4_lemma23_}  into \eqref{beforeclaims34_lemma23_} with $\delta_0$ small, we  have
\begin{align*}
 &\| \na^3\vp\|_w^2 + \| \na^3\ps\|^2 + \int_0^t \| \na^4 \vp\|_w^2  + \ep \int_0^t\|\na^4\ps\|^2 \\
 &\le  C ( \| \na \ps_0\|_{2,w}^2 + \| \ps_0\|^2 + \| \vp_0\|_{3,w}^2). 
\end{align*}
So from \eqref{claim4_lemma23_},  we have
\begin{align*}
\int \fn{ |\na^3 \ps|^2} + \int_0^t \int \fn{ |\na^3 \ps|^2} + 
 \ep\int_0^t \int \fn{ |\na^4 \ps|^2} \le C ( \|  \vp_0\|_{3,w}^2   + \| \na \ps_0\|_{2,w}^2
+ \| \ps_0\|^2).
\end{align*}   This proves Lemma \ref{lemma23_} for the case $k=3$.
 \end{proof}
 
Finally, we obtain Proposition  \ref{uniform_}. Indeed, by
  adding  Lemma \ref{lemma01_} and Lemma \ref{lemma23_} for $k=2,3$, we have 
  \begin{equation*}\begin{split}
&  
M(T)
+ \int_0^T  \sum_{l = 1}^4 \| \na^{l} \vp\|_w^2
+\int_0^T   \sum_{l = 1}^3\| \na^l \ps\|_w^2
+\ep\int_0^T     \| \na^{4} \ps\|_w^2
 \le
 C M(0).
\end{split}\end{equation*}  

{\Large \section*{acknowledgement}}
The work of MC was supported by NRF-2018R1A1A3A04079376.
The work of KC was supported by NRF-2018R1D1A1B07043065 and by the POSCO Science Fellowship of POSCO TJ Park Foundation.

\end{document}